\documentclass[11pt]{amsart}

\usepackage{enumerate,graphicx}
\usepackage[latin1]{inputenc}

\usepackage{amsmath,amsthm}
\usepackage{paralist,mathtools}

\usepackage{eucal,mathrsfs,yhmath}

\DeclareFontFamily{U}{matha}{\hyphenchar\font45}
\DeclareFontShape{U}{matha}{m}{n}{
      <5> <6> <7> <8> <9> <10> gen * matha
      <10.95> matha10 <12> <14.4> <17.28> <20.74> <24.88> matha12
      }{}
\DeclareSymbolFont{matha}{U}{matha}{m}{n}
\DeclareFontFamily{U}{mathx}{\hyphenchar\font45}
\DeclareFontShape{U}{mathx}{m}{n}{
      <5> <6> <7> <8> <9> <10>
      <10.95> <12> <14.4> <17.28> <20.74> <24.88>
      mathx10
      }{}
\DeclareSymbolFont{mathx}{U}{mathx}{m}{n}

\DeclareMathSymbol{\obot}         {2}{matha}{"6B}
\DeclareMathSymbol{\bigobot}       {1}{mathx}{"CB}

\newcommand{\secref}[1]{Section~\ref{#1}}
\newcommand{\figref}[1]{Figure~\ref{#1}}
\newcommand{\defref}[1]{Definition~\ref{#1}}
\newcommand{\crlref}[1]{Corollary~\ref{#1}}
\renewcommand{\eqref}[1]{Equation~(\ref{#1})}
\newcommand{\thmref}[1]{Theorem~\ref{#1}}
\newcommand{\prpref}[1]{Proposition~\ref{#1}}
\newcommand{\lmaref}[1]{Lemma~\ref{#1}}
\newcommand{\rmkref}[1]{Remark~\ref{#1}}

\newcommand\numberthis{\addtocounter{equation}{1}\tag{\theequation}}

\newcommand{\NN}{\mathbb N}
\newcommand{\ZZ}{\mathbb Z}
\newcommand{\RR}{\mathbb R}
\newcommand{\CC}{\mathbb C}
\newcommand{\FF}{\mathbb F}
\newcommand{\HH}{\mathbb H}

\newcommand{\spc}{\mathrm}
\newcommand{\group}{\mathfrak}
\newcommand{\complex}{\mathcal}

\newcommand{\real}{\mathrm{Re}}
\newcommand{\imaginary}{\mathrm{Im}}
\newcommand{\image}{\mathrm{im}\,}
\newcommand{\kernel}{\mathrm{ker}\,}

\newcommand{\vol}{\mathrm{vol}\,}
\newcommand{\grad}{\mathrm{grad}\,}

\newcommand{\GL}{\mathrm{GL}}

\newcommand{\id}{\mathrm{id}}

\newcommand{\ijk}{ijk}

\newcommand{\kk}{\mathfrak{K}}

\newcommand{\monodromy}{\mathfrak M}
\newcommand{\unitcircle}{\mathbb S}

\newcommand{\union}{\mathop{\cup}}
\newcommand{\intersection}{\mathop{\cap}}
\newcommand{\disjointunion}{\mathop{\sqcup}}

\newcommand{\dprime}{{\prime\!\prime}}	

\numberwithin{equation}{section}  

\newtheoremstyle{style1}% name of the style to be used
  {5pt}% measure of space to leave above the theorem. E.g.: 3pt
  {5pt}% measure of space to leave below the theorem. E.g.: 3pt
  {\it}% name of font to use in the body of the theorem
  {}% measure of space to indent
  {\bf\scshape}% name of head font
  {.}% punctuation between head and body
  {5pt}% space after theorem head; " " = normal interword space
  {}% Manually specify head
\newtheoremstyle{style2}% name of the style to be used
  {3pt}% measure of space to leave above the theorem. E.g.: 3pt
  {3pt}% measure of space to leave below the theorem. E.g.: 3pt
  {\it}% name of font to use in the body of the theorem
  {}% measure of space to indent
  {\bf}% name of head font
  {:}% punctuation between head and body
  {3pt}% space after theorem head; " " = normal interword space
  {}% Manually specify head
\theoremstyle{style1}
\newtheorem{definition}{Definition}
\newtheorem*{definition*}{Definition}
\newtheorem{theorem}{Theorem}
\newtheorem{proposition}{Proposition}
\newtheorem{corollary}{Corollary}
\newtheorem{lemma}{Lemma}
\theoremstyle{style2}
\newtheorem{remark}{Remark}
\newtheorem*{example}{Example}

\usepackage[%from peter schroeder
  breaklinks, %make formatting of line broken links work
  colorlinks=true, %indicated clickable things in blue
  linkcolor=black, %removes all the gaudy colors that come with the previous choice
  anchorcolor=black,
  citecolor=black,
  filecolor=black,
  menucolor=black,
  urlcolor=blue,
  bookmarks,
  bookmarksnumbered,
  hyperfootnotes=false,
  pdfpagelabels,
  pdfstartview=FitH, %matter of taste what you want to force here
  pdfpagemode=UseOutlines, %Open with bookmarks side pane
  pdfnewwindow=true,
  pagebackref=false, %else formatting can be quite screwed up in bib
  bookmarksopen=false,
  bookmarks=true
]{hyperref}
\usepackage[figure]{hypcap}

\begin{document}

\title[Complex Line Bundles over Simplicial Complexes and their Applications]{Complex Line Bundles over Simplicial Complexes and their Applications}

\author{Felix Kn\"oppel}
\author{Ulrich Pinkall}

\address{Felix Kn\"oppel\\
  Technische Universit\"at Berlin\\Institut f\"ur Mathematik\\
  Stra{\ss}e des 17.\ Juni 136\\
  10623 Berlin\\ Germany}

\address{Ulrich Pinkall\\
  Technische Universit\"at Berlin\\Institut f\"ur Mathematik\\
  Stra{\ss}e des 17.\ Juni 136\\
  10623 Berlin\\ Germany}

\email{knoeppel@math.tu-berlin.de, pinkall@math.tu-berlin.de}

\date{\today}

\begin{abstract}
Discrete vector bundles are important in Physics and recently found remarkable applications in Computer Graphics. This article approaches discrete bundles from the viewpoint of Discrete Differential Geometry, including a complete classification of discrete vector bundles over finite simplicial complexes. In particular, we obtain a discrete analogue of a theorem of Andr\'e Weil on the classification of hermitian line bundles. Moreover, we associate to each discrete hermitian line bundle with curvature a unique piecewise-smooth hermitian line bundle of piecewise constant curvature. This is then used to define a discrete Dirichlet energy which generalizes the well-known cotangent Laplace operator to discrete hermitian line bundles over Euclidean simplicial manifolds of arbitrary dimension.
\end{abstract}

\thanks{Both authors supported by DFG SFB/TRR 109 ``Discretization in Geometry and Dynamics''.}

\maketitle

\section{Introduction}
Vector bundles are fundamental objects in Differential Geometry and play an important role in Physics \cite{B85}. The Physics literature is also the main place where discrete versions of vector bundles were studied: First, there is a whole field called Lattice Gauge Theory where numerical experiments concerning connections in bundles over discrete spaces (lattices or simplicial complexes) are the main focus. Some of the work that has been done in this context is quite close to the kind of problems we are going to investigate here \cite{CH11,CH12,HS12}.

Vector bundles make their most fundamental appearance in Physics in the form of the complex line bundle whose sections are the wave functions of a charged particle in a magnetic field. Here the bundle comes with a connection whose curvature is given by the magnetic field \cite{B85}. There are situations where the problem itself suggests a natural discretization: The charged particle (electron) may be bound to a certain arrangement of atoms. Modelling this situation in such a way that the electron can only occupy a discrete set of locations then leads to the ``tight binding approximation'' \cite{KS96,AO03,SK08}.

Recently vector bundles over discrete spaces also have found striking applications in Geometry Processing and Computer Graphics. We will describe these in detail in \secref{sec:applications}.

In order to motivate the basic definitions concerning vector bundles over simplicial complexes let us consider a smooth manifold \(\tilde{\spc M}\) that comes with smooth triangulation (\figref{fig:smooth-triangulation}).

\begin{figure}[h]
\centering
\includegraphics[width=.5\columnwidth]{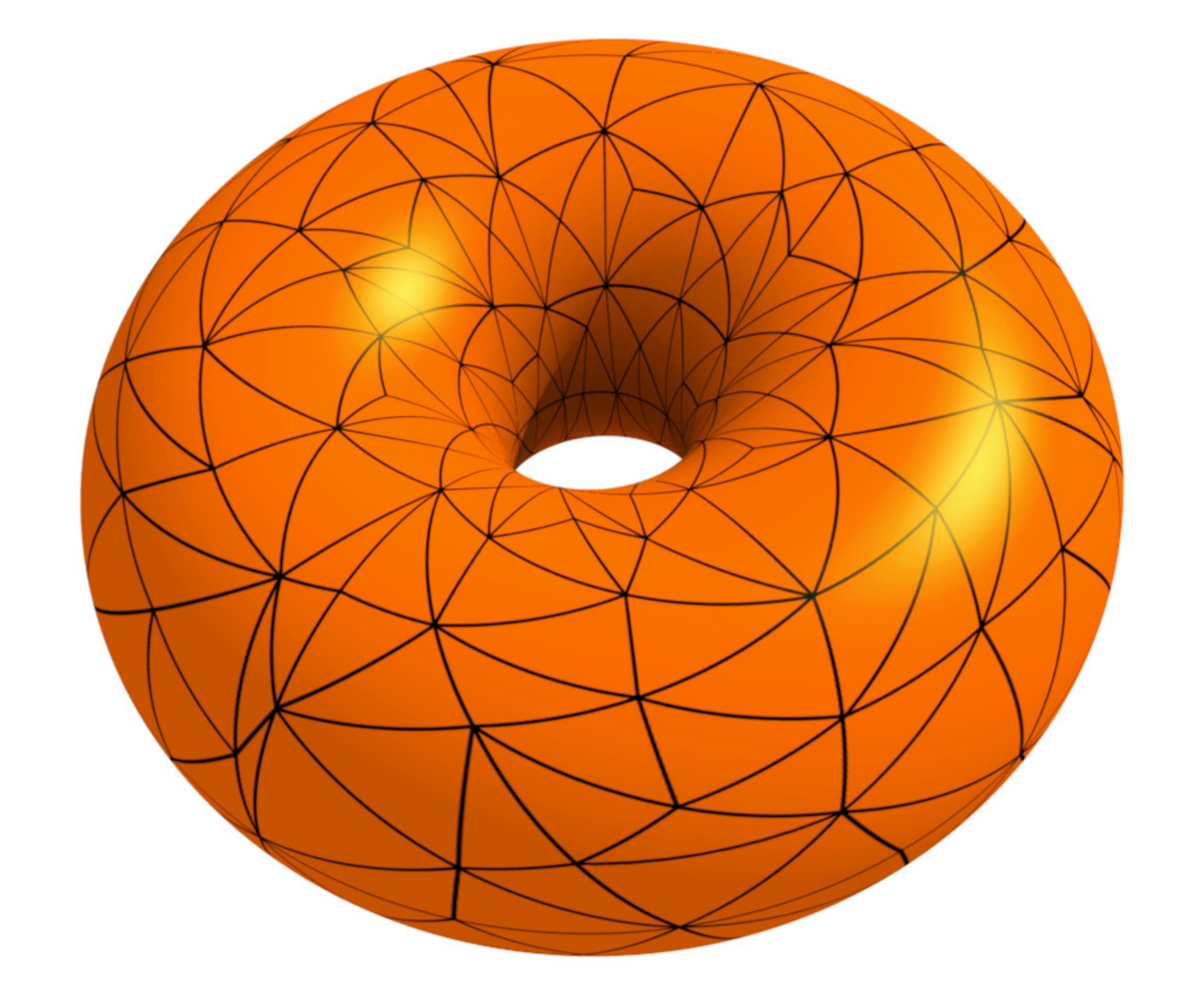}
\caption{A smooth triangulation of a manifold.}
\label{fig:smooth-triangulation}
\end{figure}

Let \(\tilde{\spc E}\) be a smooth vector bundle over \(\tilde{\spc M}\) of rank \(\kk\). Then we can define a discrete version \(\spc E\) of \(\tilde{\spc E}\) by restricting \(\tilde{\spc E}\) to the vertex set \(\mathcal V\) of the triangulation. Thus \(\spc E\) assigns to each vertex \(i\in \mathcal V\) the \(\kk\)-dimensional real vector space \(\spc E_i:=\tilde{\spc E}_i\). This is the way vector bundles over simplicial complexes are defined in general: Such a bundle \(\spc E\) assigns to each vertex \(i\) a \(\kk\)-dimensional real vector space \(\spc E_i\) in such a way that \(\spc E_i\intersection \spc E_j = \emptyset\) for \(i\neq j\).

So far the notion of a discrete vector bundle is completely uninteresting mathematically: The obvious definition of an isomorphism between two such bundles \(\spc E\) and \(\hat{\spc E}\) just would require a vector space isomorphism \(f_i\colon \spc E_i\to \hat{\spc E}_i\) for each vertex \(i\). Thus, unless we put more structure on our bundles, any two vector bundles of the same rank over a simplicial complex are isomorphic.

Suppose now that \(\tilde{\spc E}\) comes with a connection \(\nabla\). Then we can use the parallel transport along edges \(\ij\) of the triangulation to define vector space isomorphisms
\[
\eta_{\ij}\colon\tilde{\spc E}_i\to \tilde{\spc E}_j
\]

This leads to the standard definition of a connection on a vector bundle over a simplicial complex: Such a connection is given by a collection of isomorphisms \(\eta_{\ij}\colon \spc E_i\to \spc E_j\) defined for each edge \(\ij\) such that 
\[
\eta_{ji}=\eta_{\ij}^{-1}.
\]
Now the classification problem becomes non-trivial because for an isomorphism \(f\) between two bundles \(\spc E\) and \(\hat{\spc E}\) with connection we have to require compatibility with the transport maps \(\eta_{\ij}\):
\[
f_j\circ \eta_{\ij}=\hat{\eta}_{\ij}\circ f_i.
\]

Given a connection \(\eta\) and a closed edge path \(\gamma=e_\ell\cdots e_1\) (compare \secref{sec:classifying-bundles})
of the simplicial complex we can define the monodromy \(P_\gamma \in \mathrm{Aut}(\spc E_i)\) around \(\gamma\) as
\[
P_\gamma = \eta_{e_\ell}\circ \ldots \circ \eta_{e_1}.
\]

In particular the monodromies around triangular faces of the simplicial complex provide an analog for the smooth curvature in the discrete setting. In \secref{sec:classifying-bundles} we will classify vector bundles with connection in terms of their monodromies.

Let us look at the special case of a rank \(2\) bundle \(\spc E\) that is oriented and comes with a Euclidean scalar product. Then the \(90^{\circ}\)-rotation in each fiber makes it into \(1\)-dimensional complex vector space, so we effectively are dealing with a hermitian complex line bundle. If \(\ijk\) is an oriented face of our simplicial complex, the monodromy \(P_{\partial\,\ijk}\colon \spc E_i\to \spc E_i\) around the triangle \(\ijk\) is multiplication by a complex number \(h_{\ijk}\) of norm one. Writing \(h_{\ijk}=e^{\imath\alpha_{\ijk}}\) with \(-\pi<\alpha_{\ijk}\leq \pi\) we see that this monodromy can also be interpreted as a real curvature \(\alpha_{\ijk}\in (-\pi,\pi]\). It thus becomes apparent that the information provided by the connection \(\eta\) cannot encode any curvature that integrated over a single face is larger than \(\pm\pi\). This can be a serious restriction for applications: We effectively see a cutoff for the curvature that can be contained in a single face.

Remember however our starting point: We asked for structure that can be naturally transferred from the smooth setting to the discrete one. If we think again about a triangulated smooth manifold it is clear that we can associate to each two-dimensional face \(\ijk\) the integral \(\Omega_{\ijk}\) of the curvature \(2\)-form over this face. This is just a discrete \(2\)-form in the sense of discrete exterior calculus \cite{DK08}. Including this discrete curvature \(2\)-form with the parallel transport \(\eta\) brings discrete complex line bundles much closer to their smooth counterparts:

{\bf Definition:} {\it A hermitian line bundle with curvature over a simplicial complex \(\complex X\) is a triple \((\spc E,\eta,\Omega)\). Here \(\spc E\) is complex hermitian line bundle over \(\complex X\), for each edge \(\ij\) the maps \(\eta_{\ij}\colon \spc E_i\to \spc E_j\) are unitary and the closed real-valued \(2\)-form \(\Omega\) on each face \(\ijk\) satisfies
\[
\eta_{ki}\circ \eta_{jk}\circ \eta_{\ij}= e^{\imath\Omega_{\ijk}}\,\id_{\spc E_i}.
\]}

In \secref{sec:classifying-line-bundles} we will prove for hermitian line bundles with curvature the discrete analog of a well-known theorem by Andr\'e Weil on the classification of hermitian line bundles.

In \secref{sec:index} we will define for hermitian line bundles with curvature a degree (which can be an arbitrary integer) and we will prove a discrete version of the Poincar\'e-Hopf index theorem concerning the number of zeros of a section (counted with sign and multiplicity).

Finally we will construct in \secref{sec:associated-piecewise-smooth-bundle} for each hermitian line bundle with curvature a piecewise-smooth bundle with a curvature \(2\)-form that is constant on each face. Sections of the discrete bundle can be canonically extended to sections of the piecewise-smooth bundle. This construction will provide us with finite elements for bundle sections and thus will allow us to compute the Dirichlet energy on the space of sections.

\section{Applications of Vector Bundles in Geometry Processing}
\label{sec:applications}
Several important tasks in Geometry Processing (see the examples below) lead to the problem of coming up with an optimal normalized section \(\phi\) of some Euclidean vector bundle \(\spc E\) over a compact manifold with boundary \(\spc M\). Here ``normalized section'' means that \(\phi\) is defined away from a certain singular set and where defined it satisfies $|\phi|=1$.

In all the mentioned situations \(\spc E\) comes with a natural metric connection \(\nabla\) and it turns out that the following method for finding \(\phi\) yields surprisingly good results: 

{\it Among all sections \(\psi\) of \(\spc E\) find one which minimizes \(\int_{\spc M} |\nabla \psi|^2\) under the constraint \(\int_{\spc M} |\psi|^2=1\). Then away from the zero set of \(\psi\) use \(\phi=\psi/|\psi|\).}\smallskip

The term "optimal" suggests that there is a variational functional which is minimized by \(\phi\) and this is in fact the case. Moreover, in each of the applications there are heuristic arguments indicating that \(\phi\) is indeed a good choice for the problem at hand. For the details we refer to the original papers. Here we are only concerned with the Discrete Differential Geometry involved in the discretization of the above variational problem.

\subsection{Direction Fields on Surfaces}
Here \(\spc M\) is a surface with a Riemannian metric, \(\spc E=\spc{TM}\) is the tangent bundle and \(\nabla\) is the Levi-Civita connection. \figref{fig:bunny-field} shows the resulting unit vector field \(\phi\).
\begin{figure}[h]
\centering
\includegraphics[width=.75\columnwidth]{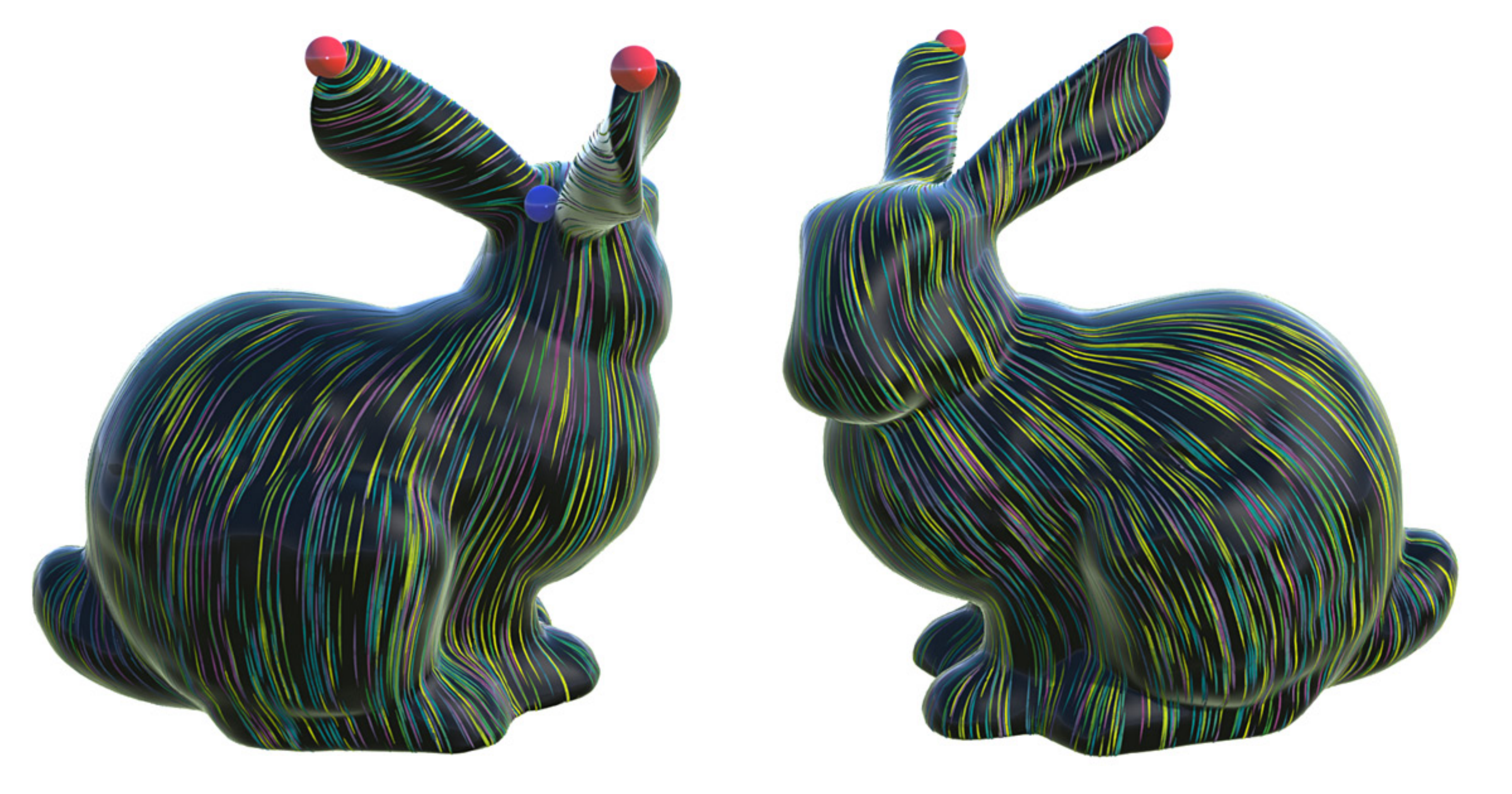}
\caption{An optimal direction field on a surface.}
\label{fig:bunny-field}
\end{figure}
If we consider \(\spc{TM}\) as a complex line bundle, normalized sections of the tensor square \(\spc L=\spc{TM} \otimes \spc{TM}\) describe unoriented direction fields on \(\spc M\). Similarly, ``higher order direction fields'' like cross fields are related to higher tensor powers of \(\spc{TM}\). Higher order direction fields also have important applications in Computer Graphics.

\subsection{Stripe Patterns on Surfaces}
\label{sec:stripe-patterns}
A {\em stripe pattern} on a surface \(\spc M\) is a map which away from a certain singular set assigns to each point \(p\in \spc M\) an element \(\phi(p)\in \unitcircle=\{z\in\CC||z|=1\}\). Such a map \(\phi\) can be used to color \(\spc M\) in a periodic fashion according to a color map that assigns a color to each point on the unit circle \(\unitcircle\). Suppose we are given a \(1\)-form \(\omega\) on \(\spc M\) that specifies a desired direction and spacing of the stripes, which means that ideally we would wish for something like \(\phi=e^{i\alpha}\) with \(d\alpha=\omega\). Then the algorithm in \cite{KC15} says that we should use a \(\phi\) that comes from taking \(\spc E\) as the trivial bundle \(\spc E=\spc M\times \CC\) and \(\nabla \psi=d\psi-i\omega \psi\). Sometimes the original data come from an unoriented direction field and (in order to obtain the \(1\)-form \(\omega\)) we first have to move from \(\spc M\) to a double branched cover \(\tilde{\spc M}\) of \(\spc M\). This is for example the case in \figref{fig:bunny-stripes}.
\begin{figure}[h]
\centering
\includegraphics[width=.75\columnwidth]{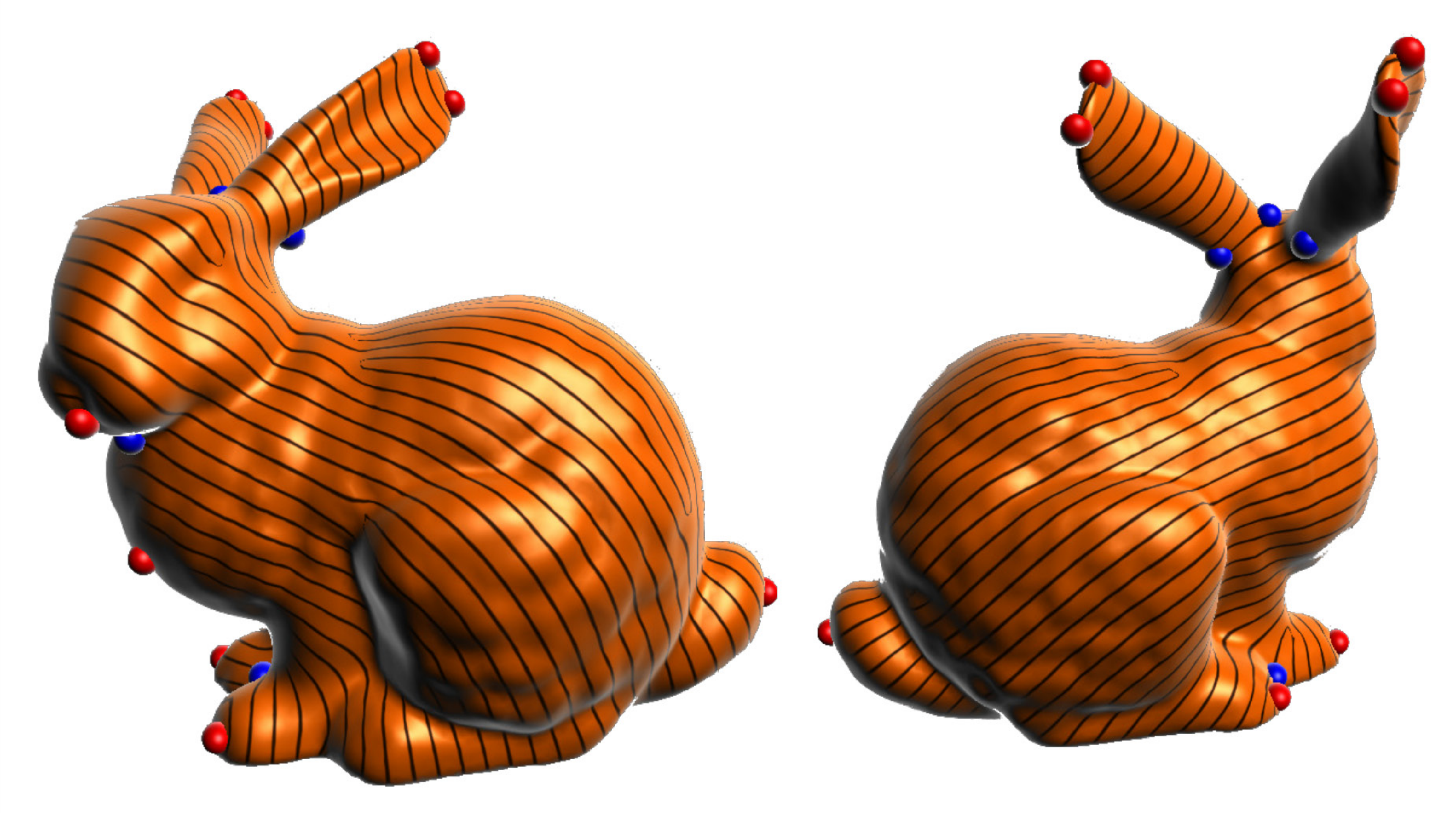}
\caption{An optimal stripe pattern aligned to an unoriented direction field.}
\label{fig:bunny-stripes}
\end{figure}

\subsection{Decomposing Velocity Fields into Fields Generated by Vortex Filaments}
The velocity fields that arise in fluid simulations quite often can be understood as a superposition of interacting vortex rings. It is therefore desirable to have an algorithm that reconstructs the underlying vortex filaments from a given velocity field. Let the velocity field \(\mathbf{v}\) on a domain \(\spc M\subset \RR^3\) be given as a \(1\)-form \(\omega=\langle \mathbf{v},\cdot\rangle\). Then the algorithm proposed in \cite{WP14} uses the function \(\phi\colon \spc M\to \CC\) that results from taking the trivial bundle \(\spc E=\spc M\times \CC\) endowed with the connection \(\nabla \psi=d\psi-i\omega \psi\). Note that so far this is just a three-dimensional version of the situation in \secref{sec:stripe-patterns}. This time however we even forget \(\phi\) in the end and only retain the zero set of \(\psi\) as the filament configuration we are looking for.

\begin{figure}[h]
\centering
\includegraphics[width=.55\columnwidth]{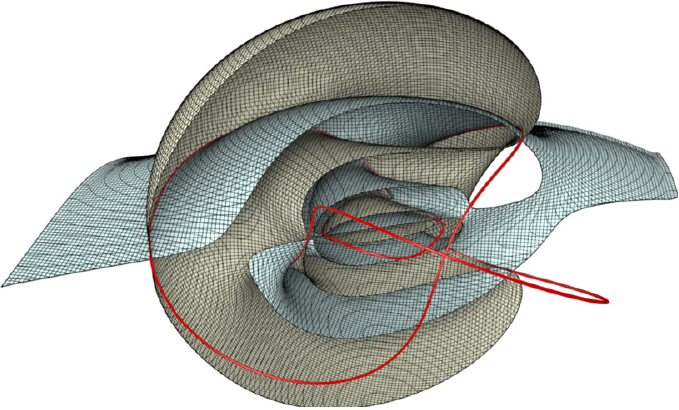}
\caption{A knotted vortex filament defined as the zero set of a complex valued function \(\psi\). It is shown as the intersection of the zero set of \(\real\,\psi\) with the zero set of \(\imaginary\,\psi\).}
\label{fig:trefoil-psi}
\end{figure}

\subsection{Close-To-Conformal Deformations of Volumes}
Here the data are a domain \(\spc M\subset \RR^3\) and a function \(u\colon \spc M\to \RR\). The task is to find a map \(f\colon \spc M\to \RR^3\) which is approximately conformal with conformal factor \(e^u\), i.e. for all tangent vectors \(X\in \spc{TM}\) we want
\[
|df(X)|\approx e^u|X|.
\]
The only exact solutions of this equations are the M{\"o}bius transformations. For these we find
\[
df(X)=e^u \overline{\psi}X\psi
\]
for some map \(\psi\colon \spc M\to \HH\) with \(|\psi|=1\) which in addition satisfies
\[
d\psi(X) = -\tfrac{1}{2}(\mathrm{grad}\,u \times X)\,\psi.
\]
Note that here we have identified \(\RR^3\) with the space of purely imaginary quaternions. Let us define a connection \(\nabla\) on the trivial rank \(4\) vector bundle \(\spc M\times \HH\) by
\[
\nabla_X \psi :=d\psi(X)+\tfrac{1}{2}(\mathrm{grad}\,u \times X)\psi.
\]
Then we can apply the usual method and find a section \(\phi \colon \spc M\to \HH\) with \(|\phi|=1\). In general there will not be any \(f\colon \spc M\to \RR^3\) that satisfies
\begin{equation}
\label{eq:df}
df(X)=e^u \overline{\phi}X\phi
\end{equation}
exactly but we can always look for an \(f\) that satisfies (\ref{eq:df}) in the least squares sense. See \figref{fig:teaser} for an example.

\begin{figure}[h]
\centering
\includegraphics[width=\columnwidth]{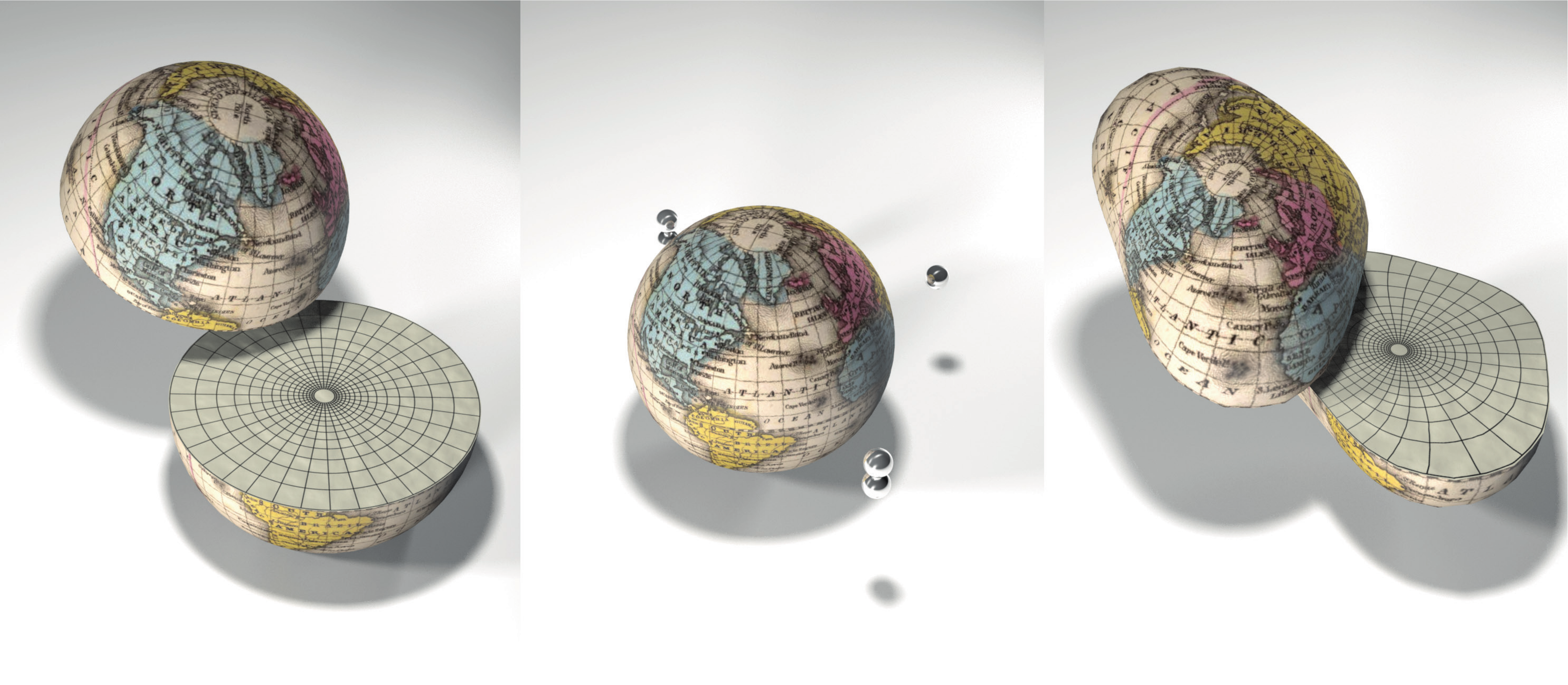}
\caption{Close-to-conformal deformation of a sphere based on a desired conformal factor specified as the potential of a collection of point charges.}
\label{fig:teaser}
\end{figure}

\section{Discrete Vector Bundles with Connection}

An {\it (abstract) simplicial complex} is a collection \(\complex X\) of finite non-empty sets such that if \(\sigma\) is an element of \(\complex X\) so is every non-empty subset of \(\sigma\) (\cite{munkres}). 

An element of a simplicial complex \(\complex X\) is called a {\it simplex} and each non-empty subset of a simplex \(\sigma\) is called a {\it face of \(\sigma\)}. The elements of a simplex are called {\it vertices} and the union of all vertices \(\mathcal V = \union_{\sigma\in \mathcal X} \sigma\) is called the {\it vertex set of \(\complex X\)}. The {\it dimension of a simplex} is defined to be one less than the number of its vertices: \(\dim \sigma := |\sigma|-1\). A simplex of dimension \(k\) is also called a {\it \(k\)-simplex}. The {\it dimension of a simplicial complex} is defined as the maximal dimension of its simplices.

To avoid technical difficulties, we will restrict our attention to {\it finite} simplicial complexes only. The main concepts are already present in the finite case. Though, the definitions carry over verbatim to infinite simplicial complexes.

\begin{definition}
Let \(\FF\) be a field and let \(\complex X\) be a simplicial complex with vertex set \(\mathcal V\). A {\it discrete \(\FF\)-vector bundle} of rank \(\kk\in\NN\) over \(\complex X\) is a map \(\pi \colon \spc E \to \mathcal V\) such that for each vertex \(i\in \mathcal V\) the {\it fiber over \(i\)} 
\[
\spc E_i := \pi^{-1}(\{i\})
\] 
has the structure of a \(\kk\)-dimensional \(\FF\)-vector space.
\end{definition}

Most of the time, we slightly abuse notation and denote a discrete vector bundle over a simplicial complex schematically by \(\spc E \to \complex X\).

The usual vector space constructions carry fiberwise over to discrete vector bundles. So we can speak of tensor products or dual bundles. Moreover, the fibers can be equipped with additional structures. In particular, a real vector bundle whose fibers are Euclidean vector spaces is called a {\it discrete Euclidean vector bundle}. Similarly, a complex vector bundle whose fibers are hermitian vector spaces is called a {\it discrete hermitian vector bundle}.  

So far discrete vector bundles are completely uninteresting from the mathematical viewpoint -- the obvious definition of an isomorphism \(f\) between two discrete vector bundles \(\spc E\) and \(\tilde{\spc E}\) just requires isomorphisms between the fibers \(f_i\colon \spc E_i \to \tilde{\spc E}_i\). Thus any two bundles of rank \(\kk\) are isomorphic. This changes if we connect the fibers along the edges by isomorphisms.

Let \(\sigma=\{i_0,\ldots,i_k\}\) be a \(k\)-simplex. We define two orderings of its vertices to be equivalent if they differ by an even permutation. Such an equivalence class is then called an {\it orientation} of \(\sigma\) and a simplex together with an orientation is called an {\it oriented simplex}. We will denote the oriented \(k\)-simplex just by the word \(i_0\cdots i_k\). Further, an oriented \(1\)-simplex is called an {\it edge}.

\begin{definition}
Let \(\spc E \to \complex X\) be a discrete vector bundle over a simplicial complex. A {\it discrete connection on \(\spc E\)} is a map \(\eta\) which assigns to each edge \(\ij\) an isomorphism \(\eta_{\ij}\colon \spc E_i \to \spc E_j\) of vector spaces such that
\[
\eta_{ji}= \eta_{\ij}^{-1}.
\]
\end{definition}

\begin{remark}
Here and in the following a morphism of vector spaces is a linear map that also preserves all additional structures - if any present. E.g., if we are dealing with hermitian vector spaces, then a morphism is a complex-linear map that preserves the hermitian metric, i.e. it is a complex linear isometric immersion.
\end{remark}

\begin{definition}
A {\it morphism of discrete vector bundles with connection} is a map \(f\colon \spc E\to \spc F\) between discrete vector bundles \(\spc E\to \complex X\) and \(\spc F\to \complex X\) with connections \(\eta\) and \(\theta\) (resp.) such that 
\begin{compactenum}[i)]
	\item for each vertex \(i\) we have that \(f(\spc E_i)\subset \spc F_i\) and the map \(f_i=\left.f\right|_{\spc E_i}\colon \spc E_i \to \spc F_i\) is a morphism of vector spaces,
	\item for each edge \(\ij\) the following diagram commutes: 
	\begin{center}
		\includegraphics[width=2.7cm]{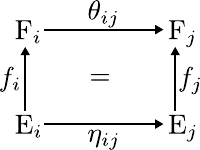},
	\end{center}i.e. \(\theta_{\ij}\circ f_i= f_j\circ\eta_{\ij}\). 
\end{compactenum}
\end{definition}
An {\it isomorphism} is a morphism which has an inverse map, which is also a morphism. Two discrete vector bundles with connection are called {\it isomorphic}, if there exists an isomorphism between them.

Again let \(\mathcal V\) denote the vertex set of \(\complex X\). A discrete vector bundle \(\spc E \to \complex X\) with connection \(\eta\) is called {\it trivial}, if it is isomorphic to the {\it product bundle} 
\[
\underline{\FF^\kk}:=\mathcal V\times\FF^\kk
\]
over \(\complex X\) equipped with the connection which assigns to each edge the identity \(\id_{\FF^\kk}\).

It is a natural question to ask how many non-isomorphic discrete vector bundles with connection exist on a given simplicial complex \(\complex X\). 

\begin{remark}
For \(k\in \NN\), the \(k\)-skeleton \(\complex X^k\) of a simplicial complex \(\complex X\) is the subcomplex that consists of all simplices \(\sigma \in \complex X\) of dimension \(\leq k\),
\[
\complex X^k := \bigl\{\sigma \in \complex X \mid \dim \sigma \leq k\bigr\}.
\]
The classification of vector bundles over \(\complex X\) only involves its \(1\)-skeleton \(\complex X^1\) and could be equally done just for discrete vector bundles over \(1\)-dimensional simplicial complexes, i.e. graphs. Later on, when we consider discrete hermitian line bundles with curvature in \secref{sec:classifying-line-bundles}, the \(2\)- and \(3\)-skeleton come into play and finally, in \secref{sec:finite-elements}, we will use the whole simplicial complex.
\end{remark}

\section{Monodromy - A Discrete Analogue of Kobayashi's Theorem}
\label{sec:classifying-bundles}

Let \(\complex X\) be a simplicial complex. Each edge \(\ij\) of \(\complex X\) has a start vertex \(s(\ij):=i\) and a target vertex \(t(\ij):=j\). A {\it edge path \(\gamma\)} is a sequence of successive edges \((e_1, \ldots, e_\ell)\), i.e. \(s(e_{k+1})=t(e_k)\) for all \(k=1,\ldots,\ell-1\), and will be denoted by the word:
\[
\gamma= e_\ell\cdots e_1.
\]
If \(i=s(e_1)\), we say that \(\gamma\) starts at \(i\), and if \(j=t(e_\ell)\) that \(\gamma\) ends at \(j\). The complex \(\complex X\) is called {\it connected}, if any two of its vertices can be joined by an edge path. From now on we will only consider connected simplicial complexes.    

Now, let \(\spc E \to \complex X\) be a discrete vector bundle with connection \(\eta\). To each edge path \(\gamma=e_\ell\cdots e_1\) from \(i\) to \(j\), we define the {\it parallel transport \(P_\gamma \colon \spc E_i \to \spc E_j\) along \(\gamma\)} by
\[
P_\gamma := \eta_{e_\ell}\circ\cdots\circ\eta_{e_1}.
\]
To each edge path \(\gamma=e_\ell\cdots e_1\) we can assign an {\it inverse path} \(\gamma^{-1}\). If \(\tilde \gamma = e_m\cdots e_{\ell+1}\) starts where \(\gamma\) ends, we can build the {\it concatenation \(\tilde\gamma\gamma\)}: With the notation \(\ij^{-1} := ji\), we have 
\[
\gamma^{-1}:=e_1^{-1}\cdots e_\ell^{-1},\quad \tilde\gamma\gamma := e_m\cdots e_{\ell+1}e_\ell\cdots e_1.
\]
Whenever \(\tilde\gamma\gamma\) is defined,
\begin{equation}\label{eq:parallel-transport-is-groupoid-morphism}
P_{\tilde\gamma\gamma} = P_{\tilde\gamma}\circ P_\gamma, \quad P_{\gamma^{-1}}= P_{\gamma}^{-1}.
\end{equation}
The elements of the fundamental group are identified with equivalence classes of {\it edge loops}, i.e. edge paths starting and ending a given {\it base vertex} \(i\) of \(\complex X\), where two such loops are identified if they differ by a sequence of {\it elementary moves} \cite{SeifertThrelfall}:  
\[
	e_\ell\cdots e_{k+1} e^{-1} e\, e_k\cdots e_1 \longleftrightarrow e_\ell\cdots e_{k+1} e_k\cdots e_1. 
\]
Now, by \eqref{eq:parallel-transport-is-groupoid-morphism}, we see that the parallel transport descends to a representation of the fundamental group \(\pi_1(\complex X^1,i)\). We encapsulate this in the following

\begin{proposition}
Let \(\spc E\to \complex X\) be a discrete vector bundle with connection over a connected simplicial complex. The parallel transport descends to a representation of the fundamental group \(\pi_1(\complex X^1, i)\): 
\[
\monodromy \colon \pi_1(\complex X^1, i) \to \mathrm{Aut}(\spc E_i),\quad [\gamma]\mapsto P_\gamma.
\] 
The representation \(\monodromy\) will be called the {\it monodromy} of the discrete vector bundle \(\spc E\).
\end{proposition}

If we change the base vertex this leads to an isomorphic representation -- an isomorphism is just given by the parallel transport \(P_\gamma\) along an edge path joining the two base vertices. Moreover, if \(f\colon \spc E\to \tilde{\spc E}\) is an isomorphism of discrete vector bundles with connection over the simplicial complex \(\complex X\), for any edge path \(\gamma= e_\ell\cdots e_1\) from a vertex \(i\) to a vertex \(j\) the following equality holds:
\[
\tilde P_\gamma = f_j \circ P_\gamma \circ f_i^{-1}.
\]
Here \(P\) and \(\tilde{P}\) denote the parallel transports of \(\spc E\) and \(\tilde{\spc E}\). Thus we obtain:
\begin{proposition}\label{prp:isomorphicmonodromy}
Isomorphic discrete vector bundles with connection have isomorphic monodromies.
\end{proposition}

In fact, the monodromy completely determines a discrete vector bundle with connection up to isomorphism, which provides a complete classification of discrete vector bundles with connection: Let \(\complex X\) be a connected simplicial complex. Let \(\spc E\to \complex X\) be a discrete \(\FF\)-vector bundle of rank \(\kk\) with connection and let \(\monodromy\colon \pi_1(\complex X^1, i)\to \mathrm{Aut}(\spc E_i)\) denote its monodromy. Any choice of a basis of the fiber \(\spc E_i\) determines a group homomorphism \(\rho\in \mathrm{Hom}\bigl(\pi_1(\complex X^1,i),\GL(\kk, \FF)\bigr)\). Any other choice of basis determines a group homomorphism \(\tilde{\rho}\) which is related to \(\rho\) by conjugation, i.e. there is \(S\in \GL(\kk, \FF)\) such that
\[
\tilde\rho([\gamma]) = S\cdot \rho([\gamma]) \cdot S^{-1} \textit{ for all } [\gamma]\in \pi_1(\complex X^1,i).
\] 
Hence the monodromy \(\monodromy\) determines a well-defined conjugacy class of group homomorphisms from \(\pi_1(\complex X^1,i)\) to \(\GL(\kk, \FF)\), which we will denote by \([\monodromy]\). The group \(\GL(\kk, \FF)\) will be referred to as the {\it structure group of \(\spc E\)}.

Let \(\mathfrak V_\FF^\kk(\complex X)\) denote the {\it set of isomorphism classes of \(\FF\)-vector bundles of rank \(\kk\) with connection over \(\complex X\)} and let \(\mathrm{Hom}\bigl(\pi_1(\complex X^1,i),\GL(\kk,\FF)\bigr)/_\sim\) denote the {\it set of conjugacy classes of group homomorphisms from the fundamental group \(\pi_1(\complex X^1,i)\) into the structure group \(\GL(\kk, \FF)\)}.

\begin{theorem}\label{thm:discretekobayashi}
\(F\colon \mathfrak V_\FF^\kk(\complex X) \to \mathrm{Hom}\bigl(\pi_1(\complex X^1,i),\GL(\kk,\FF)\bigr)/_\sim\), \([\spc E]\mapsto [\monodromy]\) is bijective.
\end{theorem}

\begin{proof}
By \prpref{prp:isomorphicmonodromy}, \(F\) is well-defined. First we show injectivity. Consider two discrete vector bundles \(\spc E\), \(\tilde{\spc E}\) over \(\complex X\) with connections \(\eta\), \(\tilde\eta\) and let \(\monodromy\), \(\tilde\monodromy\) denote their monodromies. Suppose that \([\monodromy]= [\tilde\monodromy]\). If we choose bases \(\{X^1_i,\ldots,X^\kk_i\}\) of \(\spc E_i\) and \(\{\tilde X^1_i,\ldots,\tilde X^\kk_i\}\) of \(\tilde{\spc E}_i\), then \(\monodromy\) and \(\tilde\monodromy\) are represented by group homomorphisms \(\rho,\tilde\rho\in \mathrm{Hom}\bigl(\pi_1(\complex X^1, i), \GL(\kk,\FF)\bigr)\) which are related by conjugation. Without loss of generality, we can assume that \(\rho=\tilde\rho\). Now, let \(\complex T\) be a spanning tree of \(\complex X\) with root \(i\). Then for each vertex \(j\) of \(\complex X\) there is an edge path \(\gamma_{\,i,j}\) from the root \(i\) to the vertex \(j\) entirely contained in \(\mathcal T\). Since \(\mathcal T\) contains no loops, the path \(\gamma_{\,i,j}\) is essentially unique, i.e. any two such paths differ by a sequence of elementary moves. Thus we can use the parallel transport to obtain bases \(\{X^1_j,\ldots,X^\kk_j\}\subset \spc E_j\) and \(\{\tilde X^1_j,\ldots,\tilde X^\kk_j\}\subset \tilde{\spc E}_j\) at every vertex \(j\) of \(\complex X\). With respect to these bases the connections \(\eta\) and \(\tilde\eta\) are represented by elements of \(\GL(\kk,\FF)\). By construction, for each edge \(e\) in \(\mathcal T\) the connection is represented by the identity matrix. Moreover, to each edge \(e=jk\) not contained in \(\mathcal T\) there corresponds a unique \([\gamma_e]\in \pi_1(\complex X^1, i)\). With the notation above, it is given by \(\gamma_e = \gamma^{-1}_{i,k}e\, \gamma_{\,i,j}\). In particular, on the edge \(e\) both connections are represented by the same matrix \(\rho([\gamma_{e}])=\tilde\rho([\gamma_e])\). Thus, if we define \(f\colon \spc E \to \tilde{\spc E}\) by \(f(X^m_j):=\tilde{X}^m_j\) for \(m=1,\ldots,\kk\), we obtain an isomorphism, i.e. \(\spc E\cong \tilde{\spc E}\). Hence \(F\) is injective.

To see that \(F\) is surjective we use \(\mathcal T\) to equip the product bundle \(\spc E:= \mathcal V\times \FF^\kk\) with a particular connection \(\eta\). Let \(\rho\in \mathrm{Hom}\bigl(\pi_1(\complex X^1, i),\GL(\kk,\FF)\bigr)\). If \(e\) lies in \(\mathcal T\) we set \(\eta_e = \id\) else we set \(\eta_e:= \rho([\gamma_e])\). By construction, \(F([\spc E])=[\rho]\). Thus \(F\) is surjective.
\end{proof}

\begin{remark}
Note that \thmref{thm:discretekobayashi} can be regarded as a discrete analogue of a theorem of S. Kobayashi (\cite{kobayashi,morrison}), which states that the equivalence classes of connections on principal \(G\)-bundles over a manifold \(M\) are in one-to-one correspondence with the conjugacy classes of continuous homomorphisms from the {\it path group \(\Phi(\spc M)\)} to the structure group \(G\). In fact, the fundamental group of the \(1\)-skeleton is a discrete analogue of \(\Phi(\spc M)\). 
\end{remark}

\section{Discrete Line Bundles - The Abelian Case}

In this section we want to focus on {\it discrete line bundles}, i.e. discrete vector bundle of rank \(1\). Here the monodromy descends to a group homomorphism from the closed \(1\)-chains to the multiplicative group \(\FF_\ast := \FF\setminus\{0\}\) of the underlying field. This leads to a description by discrete differential forms (\secref{sec:discrete-connection-forms}). 

Let \(\spc L\to \complex X\) be discrete \(\FF\)-line bundle over a connected simplicial complex. In this case the structure group is just \(\FF_\ast\), which is abelian. Thus we obtain
\[
\mathrm{Hom}\bigl(\pi_1(\complex X^1,i),\FF_\ast)\bigr)/_\sim = \mathrm{Hom}\bigl(\pi_1(\complex X^1, i), \FF_\ast\bigr).
\] 
\(\mathrm{Hom}\bigl(\pi_1(\complex X^1, i), \FF_\ast\bigr)\) carries a natural group structure. Moreover, the isomorphism classes of discrete line bundles over \(\complex X\) form an abelian group. The group structure is given by the tensor product: For \([\spc L], [\tilde{\spc L}]\in \mathfrak{V}_\FF^{1}(\complex X)\), we have
\[
[\spc L][\tilde{\spc L}] = [\spc L\otimes \tilde{\spc L}],\quad [\spc L]^{-1} = [\spc L^\ast].   
\]  
The identity element is given by the trivial bundle. In the following we will denote the {\it group of isomorphism classes of \(\FF\)-line bundles over \(\complex X\)} by \(\mathcal{L}_{\complex X}^\FF\). 

The map \(F\colon \mathcal{L}_{\complex X}^\FF\to \mathrm{Hom}\bigl(\pi_1(\complex X^1, i), \FF_\ast\bigr)\), \([\spc L]\mapsto [\monodromy]\) is a group homomorphism. By \thmref{thm:discretekobayashi}, \(F\) is then an isomorphism. 

Now, since \(\FF_\ast\) is abelian, each homomorphism \(\rho \in \mathrm{Hom}\bigl(\pi_1(\complex X^1, i), \FF_\ast\bigr)\) factors through the {\it abelianization} 
\[
\pi_1(\complex X^1, i)_{ab}= \pi_1(\complex X^1, i)/[\pi_1(\complex X^1, i),\pi_1(\complex X^1, i)],
\] 
i.e. for each \(\rho\in \mathrm{Hom}\bigl(\pi_1(\complex X^1, i), \FF_\ast\bigr)\) there is a unique \(\rho_{ab}\in \mathrm{Hom}\bigl(\pi_1(\complex X^1, i)_{ab}, \FF_\ast\bigr)\) such that
\[
\rho = \rho_{ab}\circ \pi_{ab}.
\]
Here \(\pi_{ab} \colon \pi_1(\complex X^1, i) \to \pi_1(\complex X^1, i)_{ab}\) denotes the canonical projection. This yields an isomorphism between \(\mathrm{Hom}\bigl(\pi_1(\complex X^1, i), \FF_\ast\bigr)\) and \(\mathrm{Hom}\bigl(\pi_1(\complex X^1, i)_{ab}, \FF_\ast\bigr)\). In particular, 
\[
\mathcal{L}_{\complex X}^\FF \cong \mathrm{Hom}\bigl(\pi_1(\complex X^1, i)_{ab}, \FF_\ast\bigr).
\] 
As we will see below, the abelianization \(\pi_1(\complex X^1, i)_{ab}\) is naturally isomorphic to the group of closed \(1\)-chains.

The {\it group of \(k\)-chains} \(\mathrm{C}_k(\complex X, \ZZ)\) is defined as the free abelian group which is generated by the \(k\)-simplices of \(\complex X\). More precisely, let \(\complex{X}^{or}_k\) denote the {\it set of oriented \(k\)-simplices of \(\complex X\)}. Clearly, for \(k>0\), each \(k\)-simplex has two orientations. Interchanging these orientations yields a fixed-point-free involution \(\rho_k \colon \complex X_k^{or} \to \complex X_k^{or}\). The group of \(k\)-chains is then explicitly given as follows: 
\[
\mathrm{C}_k(\complex X, \ZZ) := \bigl\{c \colon \complex X_k^{or} \to \ZZ \mid c\circ \rho_k = -c \bigr\}.
\]
Since simplices of dimension zero have only one orientation, \(\complex X_0^{or}=\complex X_0\). Thus,
\[
\mathrm{C}_0(\complex X, \ZZ) := \bigl\{c \colon \complex X_k^{or} \to \ZZ \bigr\}.
\]
It is common to identify an oriented \(k\)-simplex \(\sigma\) with its {\it elementary \(k\)-chain}, i.e. the chain which is \(1\) for \(\sigma\), \(-1\) for the oppositely oriented simplex and zero else. With this identification a \(k\)-chain \(c\) can be written as a formal sum of oriented \(k\)-simplices with integer coefficients: 
\[
c = \sum_{i=1}^m n_i \sigma_i,\quad n_i\in \ZZ, \, \sigma_i \in \complex X_k^{or}.
\] 
The {\it boundary operator \(\partial_k\colon \mathrm{C}_k(\complex X,\ZZ) \to \mathrm{C}_{k-1}(\complex X,\ZZ)\)} is then the homomorphism which is uniquely determined by 
\[ 
\partial_k \,i_0\cdots i_k= \sum_{j=0}^k (-1)^j\, i_0\cdots\widehat{i_j}\cdots i_k.
\]
It well-known that \(\partial_{k}\circ\partial_{k+1}\equiv 0\). Thus we get a chain complex 
\[
0 \xleftarrow{\partial_{0}}\mathrm{C}_{0}(\complex X,\ZZ)\xleftarrow{\partial_{1}}\mathrm{C}_{1}(\complex X,\ZZ)\xleftarrow{\partial_{2}}\cdots\xleftarrow{\partial_{k}}\mathrm{C}_{k}(\complex X,\ZZ)\xleftarrow{\partial_{k+1}}\cdots.
\] 
The {\it simplicial Homology groups \(\mathrm{H}_k (\complex X,\ZZ)\)} may be regarded as a measure for the deviation of exactness:
\[
\mathrm{H}_k (\complex X,\ZZ) : = \kernel\partial_k/\image\partial_{k+1}.
\]
The elements of \(\kernel\partial_k\) are called {\it \(k\)-cycles}, those of \(\image\partial_{k+1}\) are called {\it \(k\)-boundaries}.

It is a well-known fact that the abelianization of the first fundamental group is the first homology group (see e.g. \cite{hatcher}). Now, since the first homology of the \(1\)-skeleton consists exactly of all closed chains of \(\complex X\), we obtain
\[
\pi_1(\complex X^1, i)_{ab} \cong \ker \partial_1.
\]
The isomorphism is induced by the map \(\pi_1(\complex X^1, i) \to \ker\partial_1\) given by \([\gamma]\mapsto \sum_j e_j\), where \(\gamma=e_\ell\cdots e_1\). We summarize the above discussion in the following theorem.

\begin{theorem}\label{thm:classificationoflinebundles}
The group of isomorphism classes of line bundles \(\mathcal L_{\complex X}^\FF\) is naturally isomorphic to the group \(\mathrm{Hom}(\ker \partial_1, \FF_\ast)\):
\[
\mathcal L_{\complex X}^\FF \cong \mathrm{Hom}(\ker \partial_1, \FF_\ast).
\]
\end{theorem}

The isomorphism of \thmref{thm:classificationoflinebundles} can be made explicit using discrete \(\FF_\ast\)-valued \(1\)-forms associated to the connection of a discrete line bundle.

\section{Discrete Connection Forms}
\label{sec:discrete-connection-forms}

Throughout this section \(\complex X\) denotes a connected simplicial complex.

\begin{definition}
Let \(\group G\) be an abelian group. The {\it group of \(\group G\)-valued discrete \(k\)-forms} is defined as follows:
\[
\Omega^k(\complex X,\group G):=\bigl\{\omega \colon \mathrm{C}_k(\complex X) \to \group G\mid \omega \text{ group homomorphism}\bigr\}.
\]
The {\it discrete exterior derivative} \(d_k\) is then defined to be the adjoint of \(\partial_{k+1}\), i.e.
\[
d_k \colon \Omega^k(\complex X,\group G) \to \Omega^{k+1}(\complex X,\group G), \quad d_k\omega := \omega\circ\partial_{k+1}.
\]
\end{definition}

By construction, we immediately get that \(d_{k+1}\circ d_k\equiv 0\). The corresponding cochain complex is called the {\it discrete de Rahm complex with coefficients in \(\group G\)}:
\[
0 \rightarrow\Omega^0(\complex X,\group G)\xrightarrow{d_0}\Omega^1(\complex X,\group G)\xrightarrow{d_{1}}\cdots\xrightarrow{d_{k-1}}\Omega^k(\complex X,\group G)\xrightarrow{d_{k}}\cdots.
\]
In analogy to the construction of the homology groups, the {\it \(k\)-th de Rahm Cohomology group \(\mathrm{H}^k (\complex X,\group G)\) with coefficients in \(\group G\)} is defined as the quotient group 
\[
\mathrm{H}^k (\complex X,\group G) : = \kernel d_k/\image d_{k-1}.
\]
The discrete \(k\)-forms in \(\kernel d_k\) are called {\it closed}, those in \(\image d_{k-1}\) are called {\it exact}.

Now, let \(\mathfrak{C}_{\spc L}\) denote the {\it space of connections} on the discrete \(\FF\)-line bundle \(\spc L \to \complex X\):
\[
\mathfrak{C}_{\spc L} : = \bigl\{\eta \mid \eta \textit{ connection on }\spc L\bigr\}.
\] 
Any two connections \(\eta,\theta\in\mathfrak{C}_{\spc L}\) differ by a unique discrete \(1\)-form \(\omega\in \Omega^1(\complex X,\FF_\ast)\):
\[
\theta = \omega \eta.
\]
Hence the group \(\Omega^1(\mathcal X, \FF_\ast)\) acts simply transitively on the space of connections \(\mathfrak C_{\spc L}\). In particular, each choice of a {\it base connection} \(\beta \in \mathfrak C_{\spc L}\) establishes an identification
\[
\mathfrak C_{\spc L} \ni \eta =\omega\beta \longleftrightarrow \omega \in \Omega^1(\mathcal X, \FF_\ast).
\]
\begin{remark}
Note that each discrete vector bundle admits a trivial connection. To see this choose for each vertex a basis of the corresponding fiber. The corresponding coordinates establish an identification with the product bundle. Then there is a unique connection that makes the diagrams over all edges commute.  
\end{remark}

\begin{definition}
Let \(\eta\in \mathfrak{C}_{\spc L}\). A {\it connection form representing the connection \(\eta\)} is a \(1\)-form \(\omega\in \Omega^1(\complex X, \FF_\ast)\) such that \(\eta = \omega\beta\) for some trivial base connection \(\beta\).  
\end{definition}

Clearly, there are many connection forms representing a connection. We want to see how two such forms are related. 

More generally, two connections \(\eta\) and \(\theta\) in \(\mathfrak{C}_{\spc L} \) lead to isomorphic discrete line bundles if and only if for each fiber there is a vector space isomorphism \(f_i\colon \spc L_i \to \spc L_i\), such that for each edge \(\ij\):
\[
\theta_{\ij}\circ f_{i} = f_{j} \circ \eta_{\ij}.
\] 
Since \(\eta_e\) and \(\theta_e\) are linear, this boils down to discrete \(\FF_\ast\)-valued functions and the relation characterizing an isomorphism becomes 
\[
\theta_{\ij} = \bigl(g_jg_i^{-1}\bigr) \eta_{\ij} = (dg)_{\ij} \eta_{\ij},
\]
i.e. \(\eta\) and \(\theta\) differ by an exact discrete \(\FF_\ast\)-valued \(1\)-form. In particular, the difference of two connection forms representing the same connection \(\eta\) is exact. 

Thus we obtain a well-defined map sending a discrete line bundle \(\spc L\) with connection to the corresponding equivalence class of connection forms 
\[
[\omega]\in \Omega^1(\complex X, \FF_\ast)/d\Omega^0(\complex X, \FF_\ast).
\] 

\begin{theorem}\label{thm:classificationbyconnectionforms}
The map \(F\colon\mathcal L_{\complex X}^\FF\to\Omega^1(\complex X, \FF_\ast)/d\Omega^0(\complex X, \FF_\ast)\), \([\spc L]\mapsto [\omega]\), where \(\omega\) is a connection form of \(\spc L\), is an isomorphism of groups.
\end{theorem}

\begin{proof}
Clearly, \(F\) is well-defined. Let \(\spc L\) and \(\tilde{\spc L}\) be two discrete complex line bundle with connections \(\eta\) and \(\theta\), respectively. If \(\beta\in \mathfrak{C}_{\spc L}\) and \(\tilde\beta\in \mathfrak{C}_{\tilde{\spc L}}\) are trivial, so is \(\beta \otimes \tilde\beta \in \mathfrak{C}_{\spc L\otimes \tilde{\spc L}}\). Hence, with \(\eta= \omega\beta\) and \(\tilde\eta= \tilde\omega\tilde\beta\), we get
\[
F([\spc L\otimes\tilde{\spc L}]) = [\omega\tilde\omega]= [\omega][\tilde\omega]= F([\spc L])F([\tilde{\spc L}]).
\]
By the preceding discussion, \(F\) is injective. Surjectivity is also easily checked.
\end{proof}

% Next we will prove that \(\Omega^1(\complex X, \FF_\ast)/d\Omega^0(\complex X, \FF_\ast)\) is isomorphic to \(\mathrm{Hom}(\ker\partial_1,\FF_\ast)\). The isomorphism is given by the identification
% \[
% \Omega^1(\complex X, \FF_\ast)/d\Omega^0(\complex X, \FF_\ast)\ni [\omega] \mapsto \left.\omega\right|_{\ker\partial_1}\in \mathrm{Hom}(\ker\partial_1,\FF_\ast). 
% \]
Next we want to prove that the map given by
\[
\Omega^1(\complex X, \FF_\ast)/d\Omega^0(\complex X, \FF_\ast)\ni [\omega] \mapsto \left.\omega\right|_{\ker\partial_1}\in \mathrm{Hom}(\ker\partial_1,\FF_\ast)
\]
is a group isomorphism. Clearly, it is a well-defined group homomorphism. We show its bijectivity in two steps. First, the surjectivity is provided by the following

\begin{lemma}\label{lma:surjectivityofrestriction}
Let \(\complex X\) be a simplicial complex and \(\group G\) be an abelian group. Then the restriction map \(\Phi\colon \Omega^k(\complex X, \group G) \to \mathrm{Hom}(\kernel \partial_k, \group G),\,\omega \mapsto \left.\omega\right|_{\kernel\partial_k}\) is surjective.
\end{lemma}

\begin{proof}
If we choose an orientation for each simplex in \(\complex X\), then \(\partial_k\) is given by an integer matrix. Now, there is a unimodular matrix \(U\) such that \(\partial_k U = (0| H)\) has Hermite normal form. Write \(U = (A| B)\), where \(\partial_k A =0\) and \(\partial_k B =H\) and let \(a_i\) denote the columns of \(A\), i.e. \(A=(a_1,\ldots,a_\ell)\). Clearly, \(a_i\in \kernel\partial_k\). Moreover, if \(c\in \kernel\partial_k\), then \(0= \partial_k c= (0|H) U^{-1}c\). Hence \(U^{-1}c= (q,0)^\top\), \(q\in\ZZ^\ell\), and thus \(c = Aq\). Therefore \(\{a_i\mid i=1,\ldots,\ell\}\) is a basis of \(\kernel \partial_k\). Now, let \(\mu \in \mathrm{Hom}(\kernel\partial_k, \ZZ)\). A homomorphism is completely determined by its values on a basis. We define \(\omega= (\mu(a_1),\ldots,\mu(a_\ell),0\ldots,0)U^{-1}\). Then \(\omega\in \Omega^k(\complex X , \ZZ)\) and \(\omega A = (\mu(a_1),\ldots,\mu(a_\ell))\). Hence \(\Phi(\omega)= \mu\) and \(\Phi\) is surjective for forms with coefficients in \(\ZZ\). Now, let \(\group G\) be an arbitrary abelian group. And \(\mu\in\mathrm{Hom}(\kernel\partial_k, \group G)\). Now, if \(a_1,..,a_\ell\) is an arbitrary basis of \(\kernel\partial_k\), then there are forms \(\omega_1,\ldots,\omega_\ell\in \Omega^k(\complex X, \ZZ)\) such that \(\omega_i(a_j)=\delta_{\ij}\). Since \(\ZZ\) acts on \(\group G\), we can multiply \(\omega_i\) with elements \(g\in\group G\) to obtain forms with coefficients in \(\group G\). Now, set \(\omega= \sum_{i=1}^\ell \omega_i\cdot \mu(a_i)\). Then \(\omega\in\Omega^k(\complex X, \group G)\) and \(\omega(a_i)= \mu(a_i)\) for \(i=1,\ldots,\ell\). Thus \(\Phi(\omega)= \mu\). Hence \(\Phi\) is surjective for forms with coefficients in arbitrary abelian groups.
\end{proof}

The injectivity is actually easy to see: If \(\left.\omega\right|_{\ker\partial_1}=0\), we define an \(\FF_\ast\)-valued function \(f\) by {\it integration along paths}: Fix some vertex \(i\). Then
\[
f(j):= \int_{\gamma} \omega := \sum_{e\in \gamma}\omega(e),
\]
where \(\gamma\) is some path joining \(i\) to \(j\). Since \(\left.\omega\right|_{\ker\partial_1}=0\), the value \(f(j)\) does not depend on the choice of the path \(\gamma\). Moreover, \(df= \omega\). Together with \lmaref{lma:surjectivityofrestriction}, this yields the following theorem.

\begin{theorem}\label{thm:isomorphismquotienthomomorphismonkernelk1}
The map \(F\colon \Omega^1(\complex X, \FF_\ast)/d\Omega^0(\complex X, \FF_\ast)\to \mathrm{Hom}(\ker\partial_1 , \FF_\ast),\, [\omega]\mapsto \left.\omega\right|_{\ker\partial_1}\) is an isomorphism of groups.
\end{theorem}

Now, let us relate this to \thmref{thm:classificationoflinebundles}. Let \(\spc L\to \complex X\) be a line bundle with connection \(\eta\) and \(\omega\) be a connection form representing \(\eta\), i.e. \(\eta= \omega\beta\) for some trivial base connection \(\beta\). Let \([\gamma]\in \pi_1(\complex X^1, i)\), where \(\gamma=e_\ell\cdots e_1\). By linearity and since trivial connections have vanishing monodromy, we obtain
\[
\monodromy([\gamma])=\eta_{e_\ell}\circ \cdots \circ \eta_{e_1} = \omega_{e_\ell}\cdots\omega_{e_1}\cdot\beta_{e_\ell}\circ \cdots \circ \beta_{e_1}= \omega(\pi_{ab}([\gamma])) \cdot \left.\id\right|_{\spc L_i}.    
\]
Hence, by the uniqueness of \([\monodromy]_{ab}\), we obtain the following theorem that brings everything nicely together.
\begin{theorem}\label{thm:connectionformsandmonodromy}
Let \(\spc L \to \complex X\) be a line bundle with connection \(\eta\). Let \(\monodromy\) denote its monodromy and let \(\omega\) be some connection form representing \(\eta\). Then, with the identifications above,  
\[
[\monodromy]_{ab} = [\omega].
\]
\end{theorem}

\section{Curvature - A Discrete Analogue of Weil's Theorem}
\label{sec:classifying-line-bundles}

In this section we describe complex and hermitian line bundles by their curvature. For the first time we use more than the \(1\)-skeleton.

Let \(\complex X\) be a connected simplicial complex and \(\group G\) an abelian group. Since \(d^2=0\), the exterior derivative descends to a well-defined map on \(\Omega^k(\complex X, \group G)/d\Omega^{k-1}(\complex X, \group G)\), which again will be denoted by \(d\). Explicitly,
\[
d\colon \Omega^k(\complex X, \group G)/d\Omega^{k-1}(\complex X, \group G)\to \Omega^{k+1}(\complex X, \group G), \quad [\omega] \mapsto d\omega. 
\]
\begin{definition}
The {\it \(\FF_\ast\)-curvature} of a discrete \(\FF\)-line bundle \(\spc L \to \complex X\) is the discrete \(2\)-form \(\Omega\in\Omega^2(\complex X, \FF_\ast)\) given by
\[
\Omega= d[\omega],
\]
where \([\omega]\in\Omega^1(\complex X, \FF_\ast)/d\Omega^0(\complex X, \FF_\ast)\) represents the isomorphism class \([\spc L]\). 
\end{definition}

\begin{remark}
Note that \(\Omega\) just encodes the parallel transport along the boundary of the oriented \(2\)-simplices of \(\complex X\) - the ``local monodromy''.
\end{remark}

From the definition it is obvious that the \(\FF_\ast\)-curvature is invariant under isomorphisms. Thus, given a prescribed \(2\)-form \(\Omega\in\Omega^2(\complex X, \FF_\ast)\), it is a natural question to ask how many non-isomorphic line bundles have curvature \(\Omega\).  

Actually, this question is answered easily: If \(d[\omega]=\Omega=d[\tilde\omega]\), then the difference of \(\omega\) and \(\tilde\omega\) is closed. Factoring out the exact \(1\)-forms we see that the space of non-isomorphic line bundles with curvature \(\Omega\) can be parameterized by the first cohomology group \(\mathrm{H}^1(\complex X, \FF_\ast)\). Furthermore, the existence of a line bundle with curvature \(\Omega\in \Omega^2(\complex X, \FF_\ast)\) is equivalent to the exactness of \(\Omega\). 

But when is a \(k\)-form \(\Omega\) exact? Certainly it must be closed. Even more, it must vanish on every closed \(k\)-chain: If \(\Omega = \image\, d\) and \(S\) is a closed \(k\)-chain, then 
\[
\Omega(S) = d\omega(S)= \omega(\partial S)= 0.
\]
For \(k=1\), as we have seen, this criterion is sufficient for exactness. For \(k>1\) this is not true with coefficients in arbitrary groups. 

\begin{example}
Consider a triangulation \(\complex X\) of the real projective plane \(\RR\mathrm{P}^2\). The zero-chain is the only closed \(2\)-chain and hence each \(\ZZ_2\)-valued \(2\)-form vanishes on every closed \(2\)-chain. But \(\mathrm{H}^2(\complex X, \ZZ_2)= \ZZ_2\) and hence there exists a non-exact \(2\)-form.
\end{example}

In the following we will see that this cannot happen for fields of characteristic zero or, more generally, for groups that arise as the image of such fields.  

Clearly, there is a natural pairing of \(\ZZ\)-modules between \(\Omega^k(\complex X,\group G)\) and \(\mathrm{C}_k(\complex X,\ZZ)\):
\[
\langle.,.\rangle \colon \Omega^k(\complex X,\group G)\times \mathrm{C}_k(\complex X,\ZZ) \to \group G, \quad (\omega,c)\mapsto \omega(c).
\]
This pairing is degenerate if and only if all elements of \(\group G\) have bounded order. In particular, if \(\group G\) is a field \(\FF\) of characteristic zero, \(\langle.,.\rangle\) yields a group homomorphism
\[
F_k\colon \mathrm{C}_k(\complex X, \ZZ) \to \mathrm{Hom}_{\FF}(\Omega^k(\complex X, \FF), \FF) = (\Omega^k(\complex X, \FF))^\ast.
\]
A basis of \(\mathrm{C}_k(\complex X, \ZZ)\) is mapped under \(F_k\) to a basis of \((\Omega^k(\complex X, \FF))^\ast\) and hence \(\mathrm{C}_k(\complex X, \ZZ)\) appears as an \(n_k\)-dimensional lattice in \((\Omega^k(\complex X, \FF))^\ast\).

Let \(d^\ast_k\) denote the adjoint of the discrete exterior derivative \(d_k\) with respect to the natural pairing between \(\Omega^k(\complex X,\FF)\) and \((\Omega^k(\complex X, \FF))^\ast\). Clearly, 
\[
d_k^\ast\circ F_k = F_k\circ \partial_{k+1}.
\] 
Now, since the simplicial complex is finite, we can choose bases of \(\mathrm{C}_k(\complex X, \ZZ)\) for all \(k\). This in turn yields bases of \((\Omega^k(\complex X,\FF))^\ast\) and hence, by duality, bases of \(\Omega^k(\complex X, \FF)\). With respect to these bases we have 
\begin{equation}\label{eq:identifications}
\mathrm{C}_k(\complex X, \ZZ) = \ZZ^{n_k} \subset \FF^{n_k}= (\Omega^k(\complex X,\FF))^\ast = \Omega^k(\complex X,\FF),
\end{equation} 
where \(n_k\) denotes the number of \(k\)-simplices. Moreover, the pairing is represented by the standard product. The operator \(d_{k-1}^\ast = \partial_k\) is then just an integer matrix and
\[
\partial_k= d_{k-1}^\top.
\]
We have \(\image d_{k-1}\perp \kernel d_{k-1}^\ast\). Moreover, by the rank-nullity theorem, 
\[
n_k= \mathrm{dim}\,\image d_{k-1}^\ast + \mathrm{dim}\,\kernel d_{k-1}^\ast= \mathrm{dim}\,\image d_{k-1} + \mathrm{dim}\,\kernel d_{k-1}^\ast.
\]
Hence, under the identifications above, we have that \(\FF^{n_k} = \image d_{k-1}\obot \kernel d_{k-1}^\ast\) (see \figref{fig:orthogonalsum}). Moreover, \(\ker\partial_k\) contains a basis of \(\kernel d_{k-1}^\ast\). From this we conclude immediately the following lemma.
\begin{figure}[t]
\begin{center}
\includegraphics[width=6.1cm]{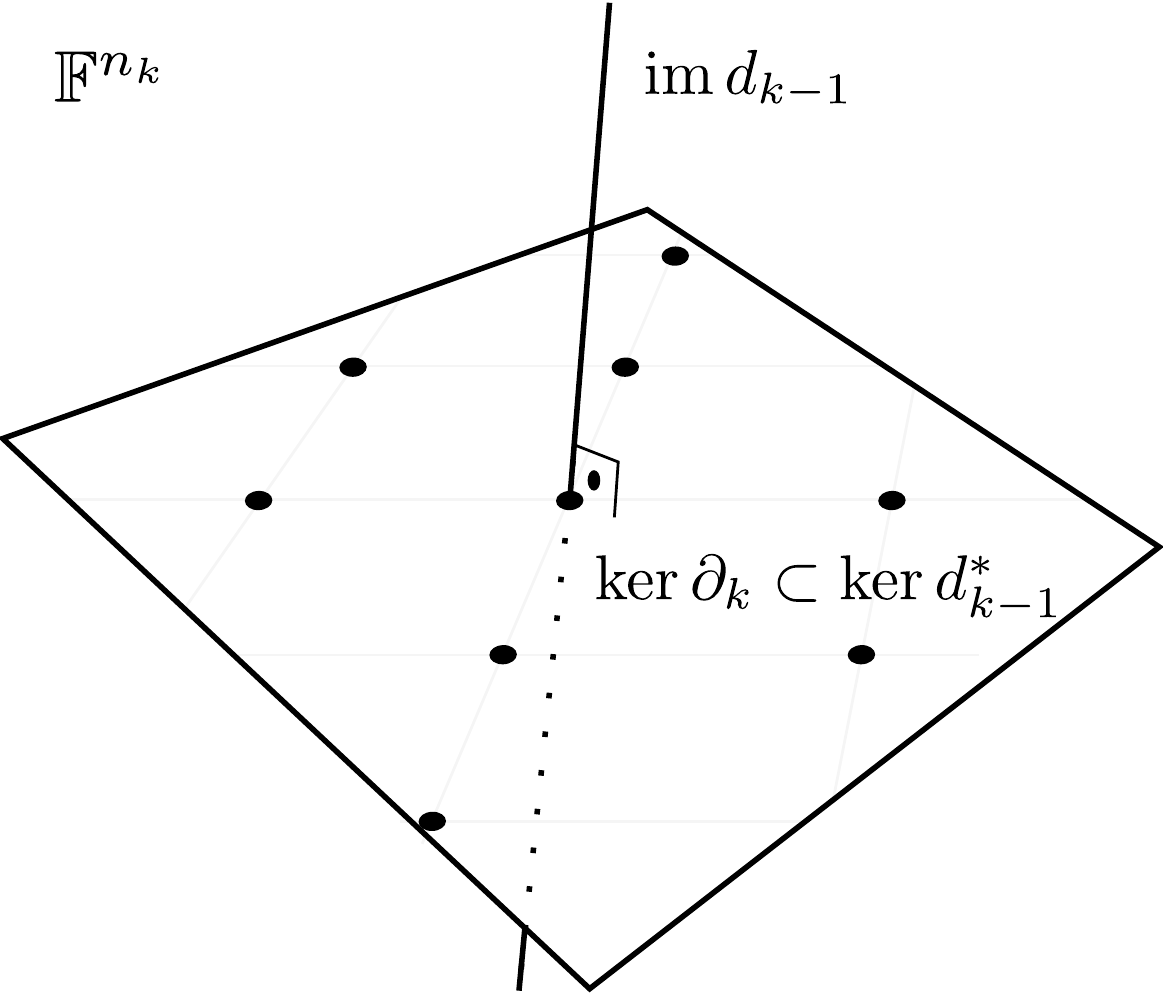}
\end{center}
\caption{With the identifications \ref{eq:identifications}, the space of \(k\)-forms becomes a direct sum of the image of \(d_{k-1}\) and the kernel of its adjoint \(d_{k-1}^\ast\), the latter of which contains the closed \(k\)-chains as a lattice.}\label{fig:orthogonalsum}	
\end{figure}

\begin{lemma}\label{lma:exactnesscriterionforfields}
Let \(\omega\in \Omega^k (\complex X, \FF)\), where \(\FF\) is a field of characteristic zero. Then 
\[
\omega \in\image d_{k-1}\Longleftrightarrow \langle \omega, c\rangle = 0 \textit{ for all } c\in \kernel \partial_k.
\]
\end{lemma}

\begin{remark}
Note, that for boundary cycles the condition is nothing but the closedness of the form \(\omega\). Thus \lmaref{lma:exactnesscriterionforfields} states that a closed form \(\omega \in \Omega^k(\complex X, \FF)\) is exact if and only if the integral over all homology classes \([c]\in \mathrm{H}_k (\complex X, \ZZ)\) vanishes.
\end{remark}

Let \(\group G\) be an abelian group. The sequence below will be referred to as the {\it \(k\)-th fundamental sequence of forms with coefficients in \(\group G\)}:
\[
\Omega^{k-1}(\complex X, \group G)\xrightarrow{d_{k-1}} \Omega^k(\complex X, \group G) \xrightarrow{\Phi_k} \mathrm{Hom}(\kernel \partial_k, \group G) \to 0,
\]
where \(\Phi_k\) denotes the restriction to the kernel of \(\partial_k\), i.e. \(\Phi_k(\omega):=\left.\omega\right|_{\ker\partial_k}\).  

Combining \lmaref{lma:surjectivityofrestriction} and \lmaref{lma:exactnesscriterionforfields} we obtain that the fundamental sequence with coefficients in a field \(\FF\) of characteristic zero is exact for all \(k>1\). This serves as an anchor point. The exactness propagates under surjective group homomorphisms.

\begin{lemma}\label{lma:propagationofexactness}
Let \(\group A \xrightarrow{f} \group B \to 0\) be a an exact sequence of abelian groups. Then, if the \(k\)-th fundamental sequence of forms is exact with coefficients in \(\group A\), so it is with coefficients in \(\group B\).
\end{lemma}

\begin{proof}
By \lmaref{lma:surjectivityofrestriction} the restriction map \(\Phi_k\) is surjective for every abelian group. It is left to check that \(\kernel \Phi_k = \image d_{k-1}\) with coefficients in \(\group B\). Let \(\Omega \in \Omega^k(\complex X, \group B)\) such that \(\Phi_k(\Omega)=0\). Since \(f\colon \group A \to \group B\) is surjective, there is a form \(\Xi\in \Omega^k(\complex X, \group A)\) such that \(\Omega = f\circ\Xi\). Since \(0=\Phi_k(\Omega) = f\circ\Phi_k(\Xi)\), we obtain that \(\Phi_k(\Xi)\) takes its values in \(\ker f\). Since \(\Phi_k\) is surjective for arbitrary groups, there is \(\Theta \in \Omega^k(\complex X, \ker f)\) such that \(\Phi_k(\Xi)= \Phi_k(\Theta)\). Hence \(\Phi_k(\Xi-\Theta)=0\). Thus there is a form \(\xi \in \Omega^{k-1}(\complex X, \group A)\) such that \(d_{k-1}\xi = \Xi- \Theta\). Now, let \(\omega := f\circ \xi \in \Omega^{k-1}(\complex X, \group B)\). Then 
\[
d_{k-1}\omega = d_{k-1} f\circ \xi = f\circ d_{k-1}\xi = f\circ (\Xi- \Theta) = f\circ \Xi =\Omega .
\]
Hence \(\kernel \Phi_k = \image d_{k-1}\) and the sequence (with coefficients in \(\group B\)) is exact.
\end{proof}

\begin{remark}\label{rmk:exactfundamentalsequence}
The map \(f\colon \CC \to \CC,\, z\mapsto \exp(2\pi \imath\, z)\) provides a surjective group homomorphism from \(\CC\) onto \(\CC_\ast\), and similarly from \(\RR\) onto \(\unitcircle\). Hence the \(k\)-th fundamental sequence of forms is exact for coefficients in \(\CC_\ast\) and in the unit circle \(\unitcircle\).
\end{remark}

\begin{remark}
The \(k\)-th fundamental sequence with coefficients in an abelian group \(\group G\) is exact if and only if \(\Omega^k(\complex X, \group G)/d\Omega^{k-1}(\complex X, \group G)\cong \mathrm{Hom}(\ker \partial_k, \group G)\). The isomorphism is induced by the restriction map \(\Phi_k\).
\end{remark}

The following corollary is a consequence of \rmkref{rmk:exactfundamentalsequence}. It nicely displays the fibration of the complex line bundles by their \(\CC_\ast\)-curvature.

\begin{corollary}\label{crl:exactsequencecorollary}
For \(\group G= \unitcircle,\, \CC_\ast\) the following sequence is exact:
\[
1\to\mathrm{H}^1(\complex X, \group G)\hookrightarrow \Omega^1(\complex X, \group G)/d\Omega^0(\complex X, \group G) \xrightarrow{d} \Omega^2(\complex X, \group G)\rightarrow \mathrm{Hom}(\kernel\partial_2,\group G) \to 1.
\]
\end{corollary}

\begin{definition}
Let \(\Omega^\ast \in \Omega^k(\complex X, \unitcircle)\). A real-valued form \(\Omega\in \Omega^2(\complex X, \RR)\) is called {\it compatible with \(\Omega^\ast\)} if \(\Omega^\ast = \exp\bigl(\imath\Omega\bigr)\). A {\it discrete hermitian line bundle with curvature} is a discrete hermitian line bundle \(\spc L\) with connection equipped with a closed \(2\)-form compatible with the \(\unitcircle\)-curvature of \(\spc L\). 
\end{definition}

For real-valued forms it is common to denote the natural pairing with the \(k\)-chains by an integral sign, i.e. for \(\omega\in\Omega^k(\complex X, \RR)\) and \(c\in \mathrm{C}_k(\complex X, \ZZ)\) we write
\[
\int_{c} \omega := \langle\omega,c\rangle= \omega(c).
\]

\begin{theorem}\label{thm:integralityofcurvature}
Let \(\spc L\) be a discrete hermitian line bundle with curvature \(\Omega\). Then \(\Omega\) is integral, i.e.
\[
\int_{c} \Omega \in 2\pi\, \ZZ, \quad \textit{ for all } c \in \ker\partial_2. 
\] 
\end{theorem}

\begin{proof}
By definition the curvature form \(\Omega\) satisfies \(\exp\bigl(\imath\Omega\bigr)= d\omega\) for some connection form \(\omega\in \Omega^1(\complex X, \unitcircle)\). Thus, if \(c\in \ker\partial_2\),
\[\exp\bigl(\imath\!\!\int_{c}\Omega\bigr) = \langle\exp(\imath\Omega),c\rangle= \langle d\omega,c\rangle= \langle \omega,\partial c\rangle = 1. 
\]   
This proves the claim.
\end{proof}

Conversely, \crlref{crl:exactsequencecorollary} yields a discrete version of a theorem of Andr\'e Weil (\cite{weil,kostant}), which states that any closed smooth integral \(2\)-form on a manifold \(\spc M\) can be realized as the curvature of a hermitian line bundle. This plays a prominent role in the process of prequantization \cite{simms}.

\begin{theorem}
If \(\Omega\in \Omega^2(\complex X, \RR)\) is integral, then there exists a hermitian line bundle with curvature \(\Omega\).
\end{theorem}

\begin{proof}
Consider \(\Omega^\ast:= \exp(i\Omega)\). Since \(\Omega\) is integral, \(\langle\Omega^\ast,c\rangle= 1\) for all \(c\in \kernel \partial_2\). By \crlref{crl:exactsequencecorollary}, there exists \(r\in \Omega^1(\complex X, \unitcircle)\) such that \(dr=\Omega^\ast\). This in turn defines a hermitian line bundle with curvature \(\Omega\).
\end{proof}

\begin{remark}
Moreover, \crlref{crl:exactsequencecorollary} shows that the connections of two such bundles differ by an element of \(\mathrm{H}^1(\complex X, \unitcircle)\). Thus the space of discrete hermitian line bundles with fixed curvature \(\Omega\) can be parameterized by \(\mathrm{H}^1(\complex X, \unitcircle)\).
\end{remark}

\section{The Index Formula for Hermitian Line Bundles}
\label{sec:index}

Before we define the degree of a discrete hermitian line bundle with curvature or the index form of a section, let us first recall the situation in the smooth setting. Therefore, let \(\spc L \to \spc M\) be a smooth hermitian line bundle with connection. Since the curvature tensor \(R^\nabla\) of \(\nabla\) is a \(2\)-form taking values in the skew-symmetric endomorphisms of \(\spc L\), it is completely described by a closed real-valued \(2\)-form \(\Omega\in \Omega^2(\spc M, \RR)\), 
\[
R^\nabla = -\imath \Omega.
\]
The following theorem shows an interesting relation between the index sum of a section \(\psi\in\Gamma(\spc L)\), the curvature \(2\)-form \(\Omega\), and the {\it rotation form \(\xi^\psi\) of \(\psi\)}. This form is defined as follows:
\[
\xi^\psi := \frac{\langle \nabla\psi, \imath\psi\rangle}{\langle \psi, \psi\rangle}.
\]

\begin{theorem}\label{thm:smoothindexformula}
Let \(\spc L \to \spc M\) be a smooth hermitian line bundle with connection and \(\Omega\) its curvature \(2\)-form. Let \(\psi \in \Gamma(\spc L)\) be a section with a discrete zero set \(Z\). Then, if \(C\) is a finite smooth \(2\)-chain such that \(\partial C\intersection Z = \emptyset\),
\[ 
2\pi\sum_{p\in {C\intersection Z}} \mathrm{ind}_p^\psi = \int_{\partial C} \xi^\psi + \int_{C} \Omega.
\]
\end{theorem}

\begin{proof}
We can assume that \(C\) is a single smooth triangle. Then we can express \(\psi\) on \(C\) in terms of a complex-valued function \(z\) and a pointwise-normalized local section \(\phi\), i.e. \(\psi= z\, \phi\). Since \(\mathrm{Im}(\tfrac{dz}{z})= d\,\mathrm{arg}(z)\), we obtain
\[
\xi^\psi= \frac{1}{|z|^2}\langle dz\, \phi + z\, \nabla \phi, \imath z\,\phi\rangle = \langle \frac{dz}{z}\, \phi, \imath\phi\rangle + \langle \nabla \phi, \imath \phi\rangle = d\,\mathrm{arg}(z) + \langle \nabla \phi, \imath \phi\rangle.
\]
Moreover, away from zeros, we have 
\[
d\langle \nabla\phi, \imath\phi\rangle = \langle R^\nabla\phi,\imath\phi\rangle + \langle \nabla\phi \wedge\imath\nabla\phi\rangle = \langle R^\nabla\phi,\imath\phi\rangle = -\Omega.
\]
Hence we obtain
\[
\int_{\partial C} \xi^\psi = \int_{\partial C} d\, \mathrm{arg}(z) + \int_{\partial C} \langle \nabla \phi, \imath \phi\rangle = 2\pi\sum_{p\in {C\intersection Z}} \mathrm{index}_p (\psi) - \int_{C} \Omega.
\]
This proves the claim.
\end{proof}

In the case that \(\spc L\) is a hermitian line bundle with connection over a closed oriented surface \(\spc M\), \thmref{thm:smoothindexformula} tells us that \(\int_{\spc M}\Omega\in 2\pi\ZZ\). This yields a well-known topological invariant - the {\it degree of \(\spc L\)}:
\[
\mathrm{deg}\bigl(\spc L\bigr):= \frac{1}{2\pi}\int_{\spc M} \Omega.
\]
From \thmref{thm:smoothindexformula} we immediately obtain the famous Poincar\'e-Hopf index theorem.

\begin{theorem}
Let \(\spc L\to \spc M\) be a smooth hermitian line bundle over a closed oriented surface. Then, if \(\psi\in \Gamma(\spc L)\) is a section with isolated zeros,
\[
\mathrm{deg}\bigl(\spc L\bigr) = \sum_{p\in \spc M} \mathrm{ind}^\psi_p.
\] 
\end{theorem}

Now, let us consider the discrete case. In general, a {\it section} of a discrete vector bundle \(\spc E\to \complex X\) with vertex set \(\mathcal V\) is a map \(\psi\colon \mathcal V \to \spc E\) such that the following diagram commutes
\begin{center}
  \includegraphics[width= 1.9cm]{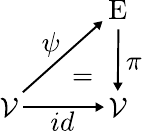},
\end{center}
i.e. \(\pi\circ\psi= id\). As in the smooth case, the space of sections of \(\spc E\) is denoted by \(\Gamma (\spc E)\).

Now, let \(\spc L\to \complex X\) be a discrete hermitian line bundle with curvature \(\Omega\) and let \(\psi\in \Gamma(\spc L)\) be a nowhere-vanishing section such that
\begin{equation}\label{eq:admissiblediscretesection}
\eta_{\ij}(\psi_i)\neq -\psi_j
\end{equation}
for each edge \(\ij\) of \(\complex X\). Here \(\eta\) denotes the connection of \(\spc L\) as usual. The {\it rotation form \(\xi^\psi\) of \(\psi\)} is then defined as follows:
\[
\xi^\psi_{\ij}:= \arg\bigl(\frac{\psi_j}{\eta_{\ij}(\psi_i)}\bigr) \in (-\pi,\pi).
\]
\begin{remark}
\eqref{eq:admissiblediscretesection} can be interpreted as the condition that no zero lies in the \(1\)-skeleton of \(\complex X\) (compare \secref{sec:finite-elements}). Actually, given a consistent choice of the argument on each oriented edge, we could drop this condition. Figuratively speaking, if a section has a zero in the \(1\)-skeleton, then we decide whether we push it to the left or the right face of the edge.     
\end{remark}
Now we can define the {\it index form} of a discrete section:

\begin{definition}
Let \(\spc L \to \complex X\) be a discrete hermitian line bundle with curvature \(\Omega\). For \(\psi\in \Gamma(\spc L)\), we define the {\it index form of \(\psi\)} by 
\[
\mathrm{ind}^\psi := \frac{1}{2\pi}\bigl(d\xi^\psi + \Omega\bigr).
\]
\end{definition}

\begin{theorem}
The index form of a nowhere-vanishing discrete section is \(\ZZ\)-valued.
\end{theorem}

\begin{proof}
Let \(\spc L\) be a discrete hermitian line bundle with curvature and let \(\eta\) be its connection. Let \(\psi \in \Gamma(\spc L)\) be a nowhere-vanishing section. Now, choose a connection form \(\omega\), i.e. \(\eta=\omega\beta\), where \(\beta\) is a trivial connection on \(\spc L\). Then we can write \(\psi\) with respect to a non-vanishing parallel section \(\phi\) of \(\beta\), i.e. there is a \(\CC\)-valued function \(z\) such that \(\psi = z\phi\). Then \(\xi^\psi_{\ij}= \arg \bigl(\frac{z_j}{\omega_{\ij}z_i}\bigr)\) and thus
\[
\exp\bigl(2\pi\imath\, d\xi^\psi_{\ijk}\bigr)=\exp\bigl(\imath\,\arg \bigl(\frac{z_i}{\omega_{ki}z_k}\bigr) + \imath\,\arg \bigl(\frac{z_j}{\omega_{\ij}z_i}\bigr) + \imath\,\arg \bigl(\frac{z_k}{\omega_{jk}z_j}\bigr)\bigr) = \frac{1}{d\omega_{\ijk}}.
\]
Thus
\[
\exp\bigl(2\pi\imath\, \mathrm{ind}^\psi_{\ijk}\bigr) = \frac{\exp\bigl(\imath\Omega_{\ijk}\bigr)}{d\omega_{\ijk}}= 1.  
\]
This proves the claim.
\end{proof}

If \(\spc L\) is a discrete hermitian line bundle with curvature \(\Omega\) over a closed oriented surface \(\complex X\), then we define the {\it degree of \(\spc L\)} just as in the smooth case:
\[
\mathrm{deg}\bigl(\spc L\bigr) := \frac{1}{2\pi} \int_{\complex X}\Omega.
\]
Here we have identified \(\complex X\) by the corresponding closed \(2\)-chain. From \thmref{thm:integralityofcurvature} we obtain the following corollary.

\begin{corollary}
The degree of a discrete hermitian line bundle with curvature is an integer:
\[
\mathrm{deg}\bigl(\spc L\bigr) \in \ZZ.
\]
\end{corollary}

The discrete Poincar\'e-Hopf index theorem follows easily from the definitions:

\begin{theorem}\label{thm:discretepoincarehopf}
Let \(\spc L\to \complex X\) be a discrete hermitian line bundle with curvature \(\Omega\) over an oriented simplicial surface. If \(\psi\in\Gamma(\spc L)\) is a non-vanishing discrete section, then 
\[
\mathrm{deg}\bigl(\spc L\bigr) = \sum_{\ijk\in \complex X} \mathrm{ind}^\psi_{\ijk}.
\]
\end{theorem}

\begin{proof}
Since the integral of an exact form over a closed oriented surface vanishes,
\[
2\pi\, \mathrm{deg}\bigl(\spc L\bigr) = \int_{\complex X} \Omega =  \int_{\complex X} d\xi^\psi + \Omega = 2\pi\, \sum_{\ijk\in \complex X} \mathrm{ind}^\psi_{\ijk},   
\]
as was claimed.
\end{proof}

\section{Piecewise-Smooth Vector Bundles over Simplicial Complexes}

It is well-known that each abstract simplicial complex \(\complex X\) has a geometric realization which is unique up to simplicial isomorphism. In particular, each abstract simplex is realized as an affine simplex. Moreover, each face \(\sigma^\prime\) of a simplex \(\sigma \in \complex X\) comes with an affine embedding 
\[
\iota_{\sigma^\prime\!\sigma}\colon \sigma^\prime\hookrightarrow \sigma.
\]
In the following, we will not distinguish between the abstract simplicial complex and its geometric realization.

\begin{definition}
A piecewise-smooth vector bundle \(\spc E\) over a simplicial complex \(\complex X\) is a topological vector bundle \(\pi\colon \spc E \to \complex X\) such that
\begin{compactenum}[a)]
	\item for each \(\sigma\in \complex X\) the restriction \(\spc E_\sigma:= \left.\spc E\right|_\sigma\) is a smooth vector bundle over \(\sigma\),
	\item for each face \(\sigma^\prime\) of \(\sigma\in \complex X\), the inclusion \(\spc E_{\sigma^\prime} \hookrightarrow \spc E_\sigma\) is a smooth embedding.
\end{compactenum}
\end{definition}

As a simplicial complex, \(\complex X\) has no tangent bundle. Nonetheless, differential forms survive as collections of smooth differential forms defined on the simplices which are compatible in the sense that they agree on common faces:

\begin{definition}
Let \(\spc E\) be a piecewise-smooth vector bundle over \(\complex X\). An \(\spc E\)-valued differential \(k\)-form is a collection \(\omega= \{\omega_\sigma\in\Omega^k(\sigma,\spc E_\sigma)\}_{\sigma\in\complex X}\) such that for each face \(\sigma^\prime\) of a simplex \(\sigma\in \complex X\) the following relation holds:
\[
\iota_{\sigma^\prime\!\sigma}^\ast \omega_{\sigma} = \omega_{\sigma^\prime},
\]
where \(\iota_{\sigma^\prime\!\sigma}\colon \sigma^\prime \hookrightarrow \sigma\) denotes the inclusion. The space of \(\spc E\)-valued differential \(k\)-forms is denoted by \(\Omega_{ps}^k(\complex X, \spc E)\).
\end{definition}

\begin{remark}
Note that a \(0\)-form defines a continuous map on the simplicial complex. Hence the definition includes functions and, more generally, sections: A piecewise-smooth section of \(\spc E\) is a continuous section \(\psi\colon \complex X \to \spc E\) such that for each simplex \(\sigma \in \complex X\) the restriction \(\psi_\sigma := \left. \psi\right|_\sigma\colon \sigma \to \spc E_\sigma\) is smooth, i.e.
\[
\Gamma_{ps}(\spc E):= \bigl\{\psi\colon \complex X \to \spc E \mid \psi_\sigma \in \Gamma(\spc E_\sigma) \textit{ for all }\sigma\in \complex X\bigr\}.
\]
\end{remark}

Since the pullback commutes with the wedge-product \(\wedge\) and the exterior derivative \(d\) of real-valued forms, we can define the wedge product and the exterior derivative of piecewise-smooth differential forms by applying it componentwise.
\begin{definition}
For \(\omega = \{\omega_\sigma\}_{\sigma\in\complex X}\in \Omega_{ps}^k(\complex X, \RR),\,\eta= \{\eta_\sigma\}_{\sigma\in\complex X}\in \Omega_{ps}^\ell(\complex X, \RR)\),
\[
	\omega\wedge\eta := \{\omega_\sigma\wedge\eta_\sigma\}_{\sigma\in\complex X}, \quad d\omega := \{d\omega_\sigma\}_{\sigma\in\complex X}. 
\]
\end{definition}
All the standard properties of \(\wedge\) and \(d\) also hold in the piecewise-smooth case.

\begin{definition}
A {\it connection} on a piecewise-smooth vector bundle \(\spc E\) over \(\complex X\) is a linear map \(\nabla \colon \Gamma_{ps}(\spc E) \to \Omega_{ps}^1(\complex X, \spc E)\) such that
\[
	\nabla(f\psi) = df\, \psi + f\,\nabla \psi,\quad \textit{for all } f\in \Omega_{ps}^0(\complex X, \RR),\,\psi\in\Gamma_{ps}(\spc E).
\]
\end{definition}

Once we have a connection on a smooth vector bundle we obtain a corresponding exterior derivative \(d^\nabla\) on \(\spc E\)-valued forms.

\begin{theorem}
Let \(\spc E\) be a piecewise-smooth vector bundle over \(\complex X\). Then there is a unique linear map \(d^\nabla\colon \Omega_{ps}^k(\complex X, \spc E) \to \Omega_{ps}^{k+1}(\complex X, \spc E)\) such that \(d^\nabla \psi = \nabla \psi\) for all \(\psi\in \Gamma_{ps}(\spc E)\) and
\[
d^\nabla (\omega\wedge\eta) = d\omega \wedge \eta + (-1)^k\omega\wedge d^\nabla\eta
\]
for all \(\omega\in \Omega_{ps}^k(\complex X, \RR)\) and \(\eta\in \Omega_{ps}^\ell(\complex X, \spc E)\).
\end{theorem}

The curvature tensor survives as a piecewise-smooth \(\mathrm{End}(\spc E)\)-valued \(2\)-form:

\begin{definition}
Let \(\spc E\to \complex X\) be a piecewise-smooth vector bundle. The endomorphism-valued curvature \(2\)-form of a connection \(\nabla\) on \(\spc E\) is defined as follows:
\[
	d^\nabla \circ d^\nabla \in \Omega_{ps}^2 (\complex X, \mathrm{End}(\spc E)).
\] 
\end{definition}

\section{The Associated Piecewise-Smooth Hermitian Line Bundle}
\label{sec:associated-piecewise-smooth-bundle}

Let \(\tilde{\spc L} \to \complex X\) be a piecewise-smooth hermitian line bundle with connection \(\nabla\) over a simplicial complex. Just as in the smooth case the endomorphism-valued curvature \(2\)-form takes values in the skew-adjoint endomorphisms and hence is given by a piecewise-smooth real-valued \(2\)-form \(\tilde\Omega\):
\[
d^{\nabla} \circ d^{\nabla} = -\imath \tilde\Omega.
\]
Since each simplex of \(\complex X\) has an affine structure, we can speak of piecewise-constant forms. 

The goal of this section will be to construct for each discrete hermitian line bundle with curvature a piecewise-smooth hermitian line bundle with piecewise-constant curvature which in a certain sense naturally contains the discrete bundle. We first prove two lemmata.

\begin{lemma}\label{lma:constantfromdiscreteforms}
To each closed discrete real-valued \(k\)-form \(\omega\) there corresponds a unique piecewise-constant piecewise-smooth \(k\)-form \(\tilde\omega\) such that 
\[
\omega(c)= \int_{c}\tilde\omega, \quad \textit{ for all } c\in \mathrm{C}_k(\complex X, \ZZ).
\]
The form \(\tilde\omega\) will be called the piecewise-smooth form associated to \(\omega\).
\end{lemma}

\begin{proof}
It is enough to consider just a single \(n\)-simplex \(\sigma\). We denote the space of piecewise-constant piecewise-smooth \(k\)-forms on \(\sigma\) by \(\Omega^k_c\) and the space of discrete \(k\)-forms on \(\sigma\) by \(\Omega^k_d\). Consider the linear map \(F\colon \Omega^k_c \to \Omega^k_d\) that assigns to \(\tilde\omega\in \Omega^k_c\) the discrete \(k\)-form given by
\[
F(\tilde\omega)_{\sigma^\prime} := \int_{\sigma^\prime}\tilde\omega .
\]
Clearly, \(F\) is injective. Moreover, since each piecewise-constant piecewise-smooth form is closed, we have \(\image \,F\subset\ker d_k\), where \(d_k\) denotes the discrete exterior derivative. Hence it is enough to show that the space of closed discrete \(k\)-forms on \(\sigma\) is of dimension \(\binom{n}{k}\). This we do by induction. Clearly, \(\dim\ker d_0 = 1 = \binom{n}{0}\). Now suppose that \(\dim\ker d_{k-1} = \binom{n}{k-1}\). By \lmaref{lma:exactnesscriterionforfields}, we have \(\ker d_k = \image d_{k-1}\). Hence
\[
\dim\ker d_k = \dim\image d_{k-1} = \dim \Omega^k_d - \dim\ker d_{k-1} = \tbinom{n+1}{k} - \tbinom{n}{k-1} = \tbinom{n}{k}. 
\]
Therefore, for each closed discrete \(k\)-form we obtain a unique piecewise-constant piecewise-smooth \(k\)-form which has the desired integrals on the \(k\)-simplices.
\end{proof}

It is a classical result that on star-shaped domains \(U\subset \RR^N\) each closed form is exact: If \(\Omega\in\Omega^k(U,\RR)\) is closed, then there exists a form \(\omega\in\Omega^{k-1}(U,\RR)\) such that \(\Omega= d\omega\). Moreover, the potential can be constructed explicitly by the map \(K\colon \Omega^k(U,\RR)\to \Omega^{k-1}(U,\RR)\) given by
\[
K(\Omega) = \sum_{i_1<\cdots < i_k} \sum_{\alpha =1}^k (-1)^{\alpha-1} \Bigl(\int_0^1 t^{k-1}\Omega_{i_1\cdots i_k}(tx)dt\Bigr)\, x_{i_\alpha}\, dx_{i_1}\wedge\ldots\wedge \widehat{dx_{i_\alpha}}\wedge\ldots\wedge dx_{i_k},
\]
where \(\Omega = \sum_{i_1<\cdots < i_k} \Omega_{i_1\cdots i_k}\, dx_{i_1}\wedge\ldots\wedge dx_{i_k}\). One directly can check that
\[
K(d\Omega) + d\, K(\Omega) = \Omega.
\]
Hence, if \(d\Omega=0\), we get \(\Omega = d\, K(\Omega)\). The same construction works for piecewise-smooth forms defined on the star of a simplex. This yields the following piecewise-smooth version of the Poincar\'e-Lemma. 

\begin{lemma}\label{lma:picewisesmoothpoincareonstars}
On the star of a simplex each closed piecewise-smooth form is exact.  
\end{lemma}

Now we are ready to prove the main result of this section.

\begin{theorem}\label{thm:existenceanduniquenessoftheassociatedbundle}
Let \(\spc L \to \complex X\) be a discrete hermitian line bundle with curvature \(\Omega\) over a simplicial complex and let \(\tilde\Omega\) be the piecewise-smooth \(2\)-form associated to \(\Omega\). Then there is a piecewise-smooth hermitian line bundle \(\tilde{\spc L}\to \complex X\) with connection \(\tilde\nabla\) of curvature \(\tilde\Omega\) such that \(\tilde{\spc L}_i = \spc L_i\) for each vertex \(i\) and the parallel transports coincide along each edge path. The bundle \(\tilde{\spc L}\) is unique up to isomorphism.
\end{theorem}

\begin{proof}
First we construct the piecewise-smooth hermitian line bundle. Let \(\spc L\to \complex X\) be a discrete hermitian line bundle with curvature \(\Omega\) and let \(\eta\) denote its connection. Let \(\mathcal V\) be the vertex set of \(\complex X\) and let \(S_i\) denote the open vertex star of the vertex \(i\). Further, since \(\Omega\) is closed, by \lmaref{lma:constantfromdiscreteforms}, there is a piecewise-constant piecewise-smooth form \(\tilde \Omega\) associated to \(\Omega\). Now, consider the set \[
\hat{\spc L} := \disjointunion_{i\in \mathcal V} S_i\times \spc L_i.
\]
Note, that \(S_i\intersection S_j \neq\emptyset\) if and only if \(\ij\) is an edge of \(\complex X\) or \(i=j\). Thus, if we set \(\eta_{ii}:= \left.\id\right|_{\spc L_i}\), we can define an equivalence relation on \(\hat{\spc L}\) as follows:
\[
(i,p,u) \sim (j,q,v) :\Longleftrightarrow p = q \textit{ and } v= \exp\bigl(-\imath \int_{\Delta_{\ij}^p}\tilde\Omega\bigr)\eta_{\ij}(u),
\]
where \(\Delta_{\ij}^p\) denotes the oriented triangle spanned by the point \(i,j\) and \(p\). Note here that \(\Delta_{\ij}^p\) is completely contained in some simplex of \(\complex X\). Let us check shortly that this really defines an equivalence relation. Here the only non-trivial property is transitivity. Therefore, let \((i,p,u)\sim (j,q,v)\) and \((j,q,v)\sim (k,r,w)\). Thus we have \(p=q=r\) and \(p\) lies in a simplex which contains the oriented triangle \(\ijk\). Clearly, the \(2\)-chain \(\Delta_{\ij}^p+ \Delta_{jk}^p + \Delta_{ki}^p\) is homologous to \(\ijk\) and since piecewise-constant forms are closed we get
\[
\int_{\Delta_{\ij}^p+\Delta_{jk}^p}\tilde\Omega = -\int_{\Delta_{ki}^p}\tilde\Omega + \int_{\ijk}\tilde\Omega = \int_{\Delta_{ik}^p}\tilde\Omega + \Omega_{\ijk}.  
\]
Hence we obtain
\begin{align*}
w&= \exp\bigl(-\imath \int_{\Delta_{jk}^p}\tilde\Omega\bigr)\eta_{jk}\Bigl(\exp\bigl(-\imath \int_{\Delta_{\ij}^p}\tilde\Omega\bigr)\eta_{\ij}(u)\Bigr)\\
&= \exp\bigl(-\imath \int_{\Delta_{\ij}^p + \Delta_{jk}^p }\tilde\Omega\bigr)\eta_{jk}\circ\eta_{\ij}(u) \\
&= \exp\bigl(-\imath \int_{\Delta_{ik}^p}\tilde\Omega - \imath \Omega_{\ijk}\bigr)\eta_{jk}\circ\eta_{\ij}(u)\\
&= \exp\bigl(-\imath \int_{\Delta_{ik}^p}\tilde\Omega\bigr)\eta_{ik}(u), 
\end{align*}
and thus \((i,p,u)\sim (k,r,w)\). Hence \(\sim\) defines an equivalence relation. One can check now that the quotient \(\tilde{\spc L} := \hat{\spc L}/_\sim\) is a piecewise-smooth line bundle over \(\complex X\). The local trivializations are then basically given by the inclusions \(S_i\times \spc L_i \hookrightarrow \tilde{\spc L}\) sending a point to the corresponding equivalence class. Moreover, all transition maps are unitary so that the hermitian metric of \(\spc L\) extends to \(\tilde{\spc L}\) and turns \(\tilde{\spc L}\) into a hermitian line bundle. Clearly, \(\left.\tilde{\spc L}\right|_{\mathcal V}= \spc L\).

Next, we need to construct the connection. Therefore we will use an explicit system of local sections: Choose for each vertex \(i\in \mathcal V\) a unit vector \(X_i\in \spc L_i\) and define \(\phi_i(p):= [i,p,X_i]\). This yields for each vertex \(i\) a piecewise-smooth section \(\phi_i\) define on the star \(S_i\). For each non-empty intersection \(S_i\intersection S_j\not=\emptyset\) we then obtain a function \(g_{\ij}\colon S_i\intersection S_j \to \unitcircle\). By the above construction, we find that, if \(\eta_{\ij}(X_i)= r_{\ij}X_{j}\),
\begin{equation}\label{eq:transitionmapsincoordinates}
g_{\ij}(p) = r_{\ij}\exp\bigl(-\imath \int_{\Delta_{\ij}^p}\tilde\Omega\bigr).
\end{equation}
Since \(\tilde\Omega\) is closed, \lmaref{lma:picewisesmoothpoincareonstars} tells us that \(\tilde\Omega|_{S_i}\) is exact. Hence there is a piecewise-smooth \(1\)-form \(\omega_i\) defined on \(S_i\) such that \(d\omega_i = \tilde\Omega|_{S_i}\). In general, the form \(\omega_i\) is only unique up to addition of an exact \(1\)-form, but among those there is a unique form \(\omega_i\) which is zero along the radial directions originating from \(i\). To see this, just choose some potential \(\tilde \omega_i\) of \(\Omega|_{S_i}\) and define a function \(f\colon S_i\to \RR\) as follows: 

For \(p\in S_i\), let \(f(p):= \int_{\gamma_i^{\,p}}\tilde{\omega}_i\), where \(\gamma_i^{\,p}\) denote the linear path from the vertex \(i\) to the point \(p\). Then \(\omega_i := \tilde\omega_i - df \) is a piecewise-smooth potential of \(\left.\Omega\right|_{S_i}\) and vanishes on radial directions. For the uniqueness, let \(\hat\omega_i\) be another such potential. Then, the difference \(\omega_i-\hat\omega_i\) is closed and hence exact on \(S_i\), i.e. there is \(f\colon S_i \to \RR\) such that \(df= \omega_i-\hat\omega_i\). Since \(df\) vanishes on radial directions \(f\) is constant on radial lines starting at \(i\) and hence constant on \(S_i\). Thus \(\omega_i=\hat\omega_i\).

Suppose that for each edge \(\ij\) the forms \(\omega_i\) and \(\omega_j\) are {\it compatible}, i.e., wherever both are defined,
\[
\imath\omega_j=\imath\omega_i+d\log g_{\ij}.
\]
Then we can define a connection \(\nabla\) as follows: Let \(\psi \in \Gamma\bigl(\tilde{\spc L}\bigr)\) and let \(X\in \spc{T}_p\sigma\) for some simplex \(\sigma\) of \(\complex X\), then there is some \(S_i\ni p\). On \(S_i\) we can express \(\psi\) with respect to \(\phi_i\), i.e. \(\psi = z\,\phi_i\) for some piecewise-smooth function \(z\colon S_i \to \CC\). Then we define
\begin{equation}\label{eq:localformofconnection}
\nabla_{X} \psi := \bigl(dz(X) - \imath\omega_i(X) z\bigr)\phi_i.
\end{equation}
In general there are several stars that contain the point \(p\). From compatibility easily follows that the definition does not depend on the choice of the vertex. Hence we have constructed a piecewise smooth connection \(\nabla\). One easily checks that \(\nabla\) is unitary and since \(d\omega_i= \tilde\Omega|_{S_i}\) we get \(d^\nabla\circ d^\nabla = -\imath \tilde\Omega\) as desired.

So it is left to check the compatibility of the forms \(\omega_{\ij}\) constructed above. Let \(\ij\) be some edge and let \(p_0\) be a point in its interior. Since \(\omega_i - \omega_j\) is closed, we can define \(\varphi\colon S_i\intersection S_j \to \RR\) by \(\varphi(p):= \int_{\gamma_p}\omega_i - \omega_j\), where \(\gamma_p\) is some path in \(S_i\intersection S_j\) from the point \(p_0\) to the point \(p\). Then, for \(p\in S_i \intersection S_j\),
\[
\int_{\Delta_p}\Omega = \int_{\partial \Delta_p}\omega_j = \int_{\ij+\gamma_j^{\,p}-\gamma_i^{\,p}}\omega_j= -\int_{\gamma_i^{\,p}}\omega_j =\int_{\gamma_i^{\,p}}\omega_i-\omega_j= \varphi(p),
\]
where as above \(\gamma_i^{\, p}\) denotes the linear path from \(i\) to \(p\) and, similarly, \(\gamma_j^{\, p}\) denotes the linear path from \(j\) to the point \(p\). From this we obtain
\[
\omega_i - \omega_j = d\varphi = d\int_{\Delta_p}\Omega
\]
and in particular \(\imath\omega_j=\imath\omega_i+d\log g_{\ij}\). This shows the existence. 

Now suppose there are two such piecewise-smooth bundles \(\tilde{\spc L}\) and \(\hat{\spc L}\) with connection \(\tilde\nabla\) and \(\hat\nabla\), respectively. We want to construct an isomorphism between \(\tilde{\spc L}\) and \(\hat{\spc L}\). Therefore we again use local systems. Explicitly, we choose a discrete direction field \(X \in \spc L\). This yields for each vertex \(i\) a vector \(X_i \in \tilde{\spc L}_i = \hat{\spc L}_i\) which extends by parallel transport along rays starting at \(i\) to a local sections \(\tilde\phi_i\) of \(\tilde{\spc L}\) and, similarly, to a local section \(\hat\phi_i\) of \(\hat{\spc L}\) defined on \(S_i\). 

Now we define \(F\colon \tilde{\spc L} \to \hat{\spc L}\) to be unique map which is linear on the fibers and satisfies \(F(\tilde\phi_i)=\hat\phi_i\) on \(S_i\). To see that \(F\) is well-defined, we need to check that it is compatible with the transition maps. But by construction both systems have equal transition maps, namely the the functions \(g_{\ij}\) from \eqref{eq:transitionmapsincoordinates} with \(r_{\ij}\) given by \(\eta_{\ij}(X_i) = r_{\ij}X_j\). Now, if \(z_i \, \tilde\phi_i=z_j \, \tilde\phi_j\), then \(z_i = z_j\, g_{\ij}\) and hence
\[
F(z_i\tilde\phi_i) = z_i\, \hat\phi_i = z_i\,g_{\ij}\hat\phi_j = z_j\,\hat\phi_j = F(z_j\, \tilde\phi_j).
\] 
Using \eqref{eq:localformofconnection} one similarly shows that \(F\circ\tilde\nabla = \hat\nabla\circ F\). Thus \(\tilde{\spc L} \cong \hat{\spc L}\).
\end{proof}

\section{Finite Elements for Hermitian Line Bundles With Curvature}
\label{sec:finite-elements}

In this section we want to present a specific finite element space on the associated piecewise-smooth hermitian line bundle of a discrete hermitian line with curvature. They are constructed from the local systems that played such a prominent role in the proof of \thmref{thm:existenceanduniquenessoftheassociatedbundle} and the usual piecewise-linear hat function.  

Let \(\tilde{\spc L}\) be the associated piecewise-smooth bundle of a discrete hermitian line bundle \(\spc L \to \complex X\) and let \(x_i \colon \complex X \to \RR\) denote the barycentric coordinate of the vertex \(i\in \mathcal V\), i.e. the unique piecewise-linear function such that \(x_i(j)=\delta_{\ij}\), where \(\delta\) is the Kronecker delta. Clearly, 
\[
\Gamma (\spc L) = \mathop{\bigoplus}_{i\in \mathcal V} \spc L_i.
\] 
To each \(X \in \spc L_i\) we now construct a piecewise-smooth section \(\tilde\psi\) as follows: First, we extend \(X\) to the vertex star \(S_i\) of the vertex \(i\) using the parallel transport along rays starting at \(i\). To get a global section \(\tilde\psi\in\Gamma_{ps}(\spc L)\) we use \(x_i\) to scale \(\tilde\phi\) down to zero on \(\partial S_i\) and extend it by zero to \(\complex X\), i.e. 
\[
\tilde\psi_p := \begin{cases}
\hfill x_i(p)\tilde\phi_p \hfill &  \textit{ for }p\in S_i,\\
\hfill 0 \hfill & \textit{ else.}
\end{cases}
\]
The above construction yields a linear map \(\iota \colon \Gamma(\spc L) \to \Gamma_{ps}(\tilde{\spc L})\). Clearly, \(\iota\) is injective - a left-inverse is just given by the restriction map 
\[
\Gamma_{ps}(\tilde{\spc L})\ni\tilde\psi \mapsto \bigl.\tilde\psi\bigr|_{\mathcal V} \in \Gamma(\spc L).
\]
\begin{definition}
The {\it space of piecewise-linear sections} is given by \(\Gamma_{pl}(\tilde{\spc L}) := \image \iota\).
\end{definition}

Thus we identified each section of a discrete hermitian line bundle with curvature with a piecewise-linear section of the associated piecewise-smooth bundle. This allows to define a discrete hermitian inner product and a discrete Dirichlet energy on \(\Gamma(\spc L)\), which is a generalization of the well-known cotangent Laplace operator for discrete functions on triangulated surfaces. Before we come to the Dirichlet energy, we define Euclidean simplicial complexes.

Similarly to piecewise-smooth forms we can define piecewise-smooth (contravariant) \(k\)-tensors as collections of compatible \(k\)-tensors: A {\it piecewise-smooth \(k\)-tensor} is a collection \(T=\{T_\sigma\}_{\sigma\in \complex X}\) of smooth contravariant \(k\)-tensors \(T_\sigma\) on \(\sigma\) such that
\[
\iota_{\sigma^\prime\sigma}^\ast T_{\sigma} = T_{\sigma^\prime},
\]
whenever \(\sigma^\prime\) is a face of \(\sigma\). A {\it Riemannian simplicial complex} is then a simplicial complex \(\complex X\) equipped with a {\it piecewise-smooth Riemannian metric}, i.e. a piecewise-smooth positive-definite symmetric \(2\)-tensor \(g\) on \(\complex X\).

The following lemma tells us that the space of constant piecewise-smooth symmetric tensors is isomorphic to functions on \(1\)-simplices.

\begin{lemma}
Let \(\complex X\) be a simplicial complex and let \(\mathcal E\) denote the set of its \(1\)-simplices. For each function \(f\colon \mathcal E \to \RR\) there exists a unique constant piecewise-smooth symmetric \(2\)-tensor \(S\) such that for each \(1\)-simplex \(e=\{i,j\}\)
\[
S_{e}(j-i,j-i) = f(e).
\]
\end{lemma}

\begin{proof}
It is enough to consider a single affine \(n\)-simplex \(\sigma=\{i_0,\ldots,i_n\}\) with vector space \(V\). Consider the map \(F\) that sends a symmetric \(2\)-tensor \(S\) on \(V\) to the function given by 
\[
F(S)(e) := S(i_k-i_j,i_k-i_j), \quad e=\{i_j,i_k\}\subset \sigma.
\]
Clearly, \(F\) is linear. Moreover, if \(Q\) denotes the quadratic form corresponding to \(S\), i.e. \(Q(X):=S(X,X)\), then 
\[
S(X,Y) = \tfrac{1}{2}\bigl(Q(X) + Q(Y) - Q(X-Y)\bigr).
\] 
Hence, from \(F(S)=0\) follows \(S=0\). Thus \(F\) is injective. Clearly, the space of symmetric bilinear forms is of dimension \(n(n+1)/2\), which equals the number of \(1\)-simplices. Thus \(F\) is an isomorphism. This proves the claim.
\end{proof}

It is also easy to write down the corresponding symmetric tensor in coordinates: Let \(\sigma = \{i_0,\ldots,i_n\}\) be a simplex. The vectors \(e_j:= i_j-i_0\), \(j=1,\ldots,n\), then yield a basis of the corresponding vector space. Let \(f\) be a function defined on the unoriented edges of \(\sigma\) and let \(x_{i_j}\) denote the barycentric coordinates of its vertices \(i_j\), then the corresponding symmetric bilinear form \smash{\(S^f_\sigma\)} is given by
\begin{equation}\label{eq:constantfromdiscretemetric}
S^f_\sigma = \sum_{1\leq j\leq n} f_{i_0 i_j}\, dx_{i_j}\otimes dx_{i_j} + \sum_{1\leq j,k\leq n,\, j\neq k} \tfrac{1}{2}\bigl(f_{i_0i_j} + f_{i_0i_k} - f_{i_ji_k}\bigr)\, dx_{i_j}\otimes dx_{i_k}.    
\end{equation}
Thus starting with a positive function \(f\), by Sylvester's criterion, it has to satisfy on each \(n\)-simplex \(n-1\) inequalities to determine a positive-definite form. If the corresponding piecewise-smooth form is positive-definite, we call \(f\) a discrete metric.

\begin{definition} 
A {\it Euclidean simplicial complex} is a simplicial complex \(\complex X\) equipped with a {\it discrete metric}, i.e. a map \(\ell\) that assigns to each \(1\)-simplex \(e\) a length \(\ell_e > 0\) such that for each simplex \(\sigma\) the symmetric tensor \(S_{\sigma}^\ell\) is positive-definite.
\end{definition}

Now, let \(\complex X\) be a Euclidean simplicial manifold of dimension \(n\) and denote by \(\complex X_n\) the set of its top-dimensional simplices. Since each simplex of \(\complex X\) is equipped with a scalar product it comes with a corresponding density and hence we know how to integrate functions over the simplices of \(\complex X\). Now, we define the {\it integral over \(\complex X\)} as follows:
\[
\int_{\complex X} f := \sum_{\sigma \in \complex X_n} \int_{\sigma} f_{\sigma},\quad f\in \Omega^0_{ps}(\complex X,\RR).
\]
Moreover, given a piecewise-smooth hermitian line bundle \(\tilde{\spc L}\to\complex X\) with curvature, then there is a canonical hermitian product \(\langle\!\langle.,.\rangle\!\rangle\) on \(\Gamma_{ps}(\tilde{\spc L})\): If \(\tilde\psi,\tilde\phi \in \Gamma_{ps}(\tilde{\spc L})\), then
\[
\langle\!\langle\tilde\psi,\tilde\phi\rangle\!\rangle = \int_{\complex X} \langle\tilde\psi,\tilde\phi\rangle.
\]

In particular, if \(\tilde{\spc L}\) is the associated piecewise-smooth bundle of a discrete hermitian line bundle \(\spc L\) with curvature \(\Omega\), then we can use \(\iota\) to pull \(\langle\!\langle.,.\rangle\!\rangle\) back to \(\Gamma(\spc L)\). Since \(\iota\) is injective this yields a hermitian product on \(\Gamma(\spc L)\).

Now we want to compute this metric explicitly in terms of given discrete data.

\begin{definition}
A piecewise-linear section \(\tilde\psi\in \Gamma_{pl}(\tilde{\spc L})\) is called {\it concentrated at a vertex \(i\)}, if it is of the form \(\tilde\psi= \iota(\psi_i)\) for some vector \(\psi_i\in \spc L_i\). 
\end{definition}

It is basically enough to compute the product of two such concentrated sections. Therefore, let \(\psi_i\in \spc L_i\) and \(\psi_j\in\spc L_j\) and let \(\tilde\psi^i\) and \(\tilde\psi^j\) denote the corresponding piecewise-linear concentrated sections. 

Now consider their product \(\langle \tilde\psi^i, \tilde\psi^j \rangle\). Clearly, this product has support \(S_i\intersection S_j\). For simplicity, we extend the discrete connection \(\eta\) to arbitrary pairs \(\ij\) in such way that \(\eta_{ii}= \id\) and \(\eta_{\ij}\colon \spc L_i \to \spc L_j\) is zero whenever \(\{i,j\}\not\in\complex X\). With this convention, \eqref{eq:transitionmapsincoordinates} yields
\begin{equation}\label{eq:basicproduct}
\langle \tilde\psi^j, \tilde\psi^i \rangle = \langle \psi_j, \eta_{\ij}(\psi_i)\rangle\, x_i x_j\, \exp\bigl(-\imath \int_{\Delta_{\ij}^p}\tilde\Omega\bigr),
\end{equation}
where \(\tilde\Omega\) denotes the constant piecewise-smooth curvature form associated to \(\Omega\).

Now, let us express the integral over \(\Delta_{\ij}^p\) on a given \(n\)-simplex. Therefore consider an \(n\)-simplex \(\sigma = \{i_0,\ldots,i_n\}\). The hat functions \(x_{i_1},\ldots,x_{i_n}\) yield affine coordinates on \(\sigma\) and we can express any \(2\)-form with respect to the basis forms \(dx_{i_j}\wedge dx_{i_k}\). One can show that
\[
\int_{\sigma^\prime} dx_{i_j}\wedge dx_{i_k}=\begin{cases}
\hfill \pm \frac{1}{2}\hfill & \textit{ for }\sigma^\prime = \pm i_j i_k i_\ell,\\
\hfill 0\hfill & \textit{ else.}
\end{cases}
\]
Thus we obtain
\[
\tilde\Omega = \sum_{1\leq j< k\leq n} 2\,\Omega_{i_0 i_j i_k}\, dx_{i_j}\wedge dx_{i_k}. 
\]
Now we want to compute the integral over the triangle \(\Delta_{i_0i_1}^p \subset \sigma\). By Stokes theorem,
\[
\int_{\Delta_{i_0i_1}^p} dx_{i_j}\wedge dx_{i_k} = \int_{i_0}^{i_1} x_{i_j}\, dx_{i_k} + \int_{i_1}^{p} x_{i_j}\, dx_{i_k}  + \int_{p}^{i_0} x_{i_j}\, dx_{i_k},
\]
where the integrals are computed along straight lines. A small computation shows
\[
\int_{\Delta_{i_0i_1}^p} dx_{i_j}\wedge dx_{i_k} = \frac{1}{2}\bigl(\delta_{1\!j}\, x_{i_k}(p) - \delta_{1\!k}\, x_{i_j}(p)\bigr),
\]
Thus, for \(j < k\), we get \(\int_{\Delta_{i_0i_1}^p} dx_{i_j}\wedge dx_{i_k} = \frac{1}{2}\delta_{1\!j}\, x_{i_k}(p)\) and hence 
\[
\int_{\Delta_{i_0i_1}^p} \tilde\Omega = \sum_{1\leq j< k\leq n} 2\, \Omega_{i_0 i_j i_k} \int_{\Delta_{i_0i_1}^p} dx_{i_j}\wedge dx_{i_k} = \sum_{j} \Omega_{i_0 i_1 i_j} x_{i_j}(p),
\] 
where we have used the convention that \(\Omega\) vanishes on all triples not representing an oriented \(2\)-simplex of \(\complex X\). With this convention \eqref{eq:basicproduct} becomes
\begin{equation}\label{eq:integrandproduct}
\langle \tilde\psi^j, \tilde\psi^i \rangle = \langle \psi_j, \eta_{\ij}(\psi_i)\rangle\,  x_i x_j\, \exp\bigl(-\imath \sum_{k} \Omega_{i j k} x_{k}\bigr).
\end{equation}

\begin{figure}[t]
\centering
\includegraphics[width=\columnwidth]{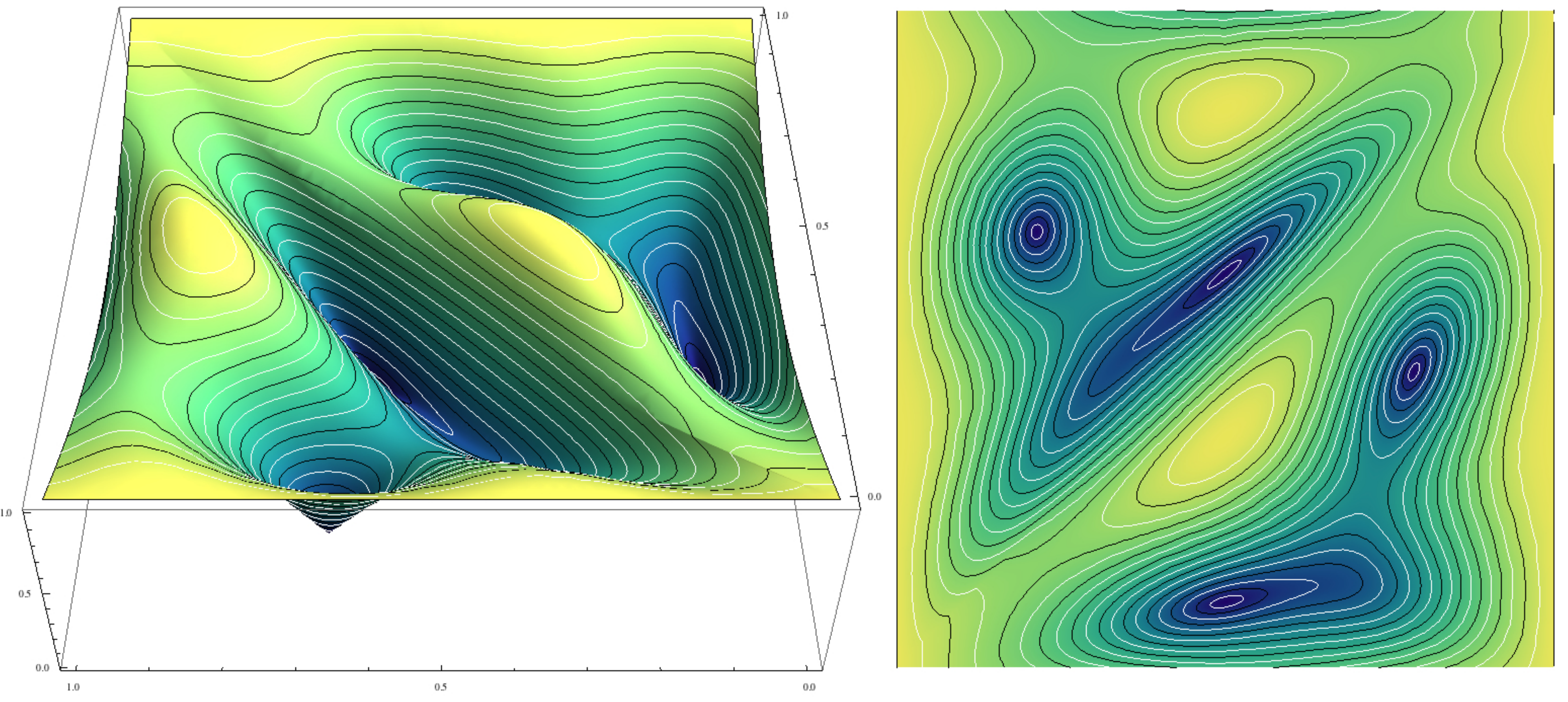}
\caption{\label{fig:continuityofsections}The graph of the norm of a piecewise-linear section of a bundle over a torus consisting of two triangles. Its two smooth parts fit continuously together along the diagonal. In this example the curvature of the bundle over each triangles is equal to \(4\pi\). Note that the section has \(4\) zeros - just as predicted.}
\end{figure}

In particular, using \eqref{eq:integrandproduct}, we can compute the norm of a piecewise-linear section \(\tilde\psi\) on a given triangle \(\ijk\). Therefore we distinguish one of its vertices, say \(i\), and write \(\tilde\psi\) with respect to a section which is radially parallel with respect to \(i\). Now, one checks that
\[
|\tilde\psi| = |c_i + x_j(c_je^{i\Omega_{\ijk}x_k}-c_i)+x_k(c_ke^{-i\Omega_{\ijk}x_j}-c_i)|,
\]
where \(c_i,c_j,c_k \in \CC\) are constants depending on the explicit form of \(\tilde\psi\). An example of the norm of a piecewise-linear section is shown in \figref{fig:continuityofsections}.

As the next proposition shows, the identification of discrete and piecewise-linear sections perfectly fits together with the definitions in \secref{sec:index}.

\begin{proposition}
Let \(\psi\in \Gamma(\spc L)\) be a discrete section and let \(\tilde\psi\in\Gamma_{pl}(\tilde{\spc L})\) be the corresponding piecewise-linear section, i.e. \(\tilde\psi = \iota(\psi)\). Then, if \(\tilde\psi\) has no zeros on edges, the discrete rotation form \(\xi^\psi\) and the piecewise-smooth rotation form \(\xi^{\tilde\psi}\) are related as follows: For each oriented edge \(\ij\),
\[\xi^\psi_{\ij} = \int_{\ij}\xi^{\tilde\psi}.\]
\end{proposition}

\begin{proof}
The claim follows easily by expressing \(\tilde\psi\) with respect to some non-vanishing parallel section along the edge \(\ij\).
\end{proof}

In particular, by \thmref{thm:smoothindexformula}, the index form of a non-vanishing section of a discrete hermitian line bundle with curvature counts the number of (signed) zeros of the corresponding piecewise-linear section of the associated piecewise-smooth bundle.

Let us continue with the computation of the metric on \(\Gamma(\spc L)\). To write down the formula we give the following definition.

\begin{definition}\label{def:canonicalmetricatoms}
Let \(\complex X\) be an \(n\)-dimensional simplicial manifold and let \(\Omega \in \Omega^2(\complex X, \RR)\). To an \(n\)-simplex \(\sigma\) and vertices \(i,j,k,l\) of \(\complex X\) we assign the value
\begin{equation*}
\Theta_{\sigma,i,j}^\Omega(k,l) := \frac{1}{\vol(\sigma)}\int_\sigma x_k x_l\, \exp\bigl(-\imath\sum_m \Omega_{\ij m} x_m\bigr),
\end{equation*}
where have chosen for integration an arbitrary discrete metric on \(\complex X\).
\end{definition}

\begin{remark}\label{rmk:independenceofchoiceofmetric}
Note that the functions \(\Theta_{\sigma,i,j}^\Omega\) are indeed well-defined. On a simplex, any two such measures induced by a discrete metric differ just by a constant.
\end{remark}

With \defref{def:canonicalmetricatoms} and \eqref{eq:integrandproduct} we obtain the following form of the metric:

\begin{theorem}[Product of Discrete Sections]\label{thm:hermitianlinebundlemetric}
Let \(\spc L\) be a discrete hermitian line bundle with curvature \(\Omega\) over an \(n\)-dimensional Euclidean simplicial manifold \(\complex X\), then the product on \(\Gamma(\spc L)\) induced by the associated piecewise-smooth hermitian line bundle is given as follows: Given two discrete sections \(\psi=\sum_i \psi_i, \phi = \sum_i \phi_i\),  
\[
\langle\!\langle \psi, \phi \rangle\!\rangle = \sum_{i,j} \mu_\Omega^{\ij}\, \langle \psi_j, \eta_{\ij}(\phi_i)\rangle, \quad \textit{where}\quad \mu_\Omega^{\ij} = \sum_{\{i,j\}\subset\sigma\in{\complex X}_n} \Theta_{\sigma,i,j}^\Omega(i,j)\, \vol(\sigma).
\]
\end{theorem}

Note that \(\Theta_{\sigma,i,j}^\Omega (k,l)\), and hence \(\mu_\Omega^{\ij}\), can be computed explicitly using Fubini's theorem and the following small lemma, which can be shown by induction. 

\begin{lemma}\label{lma:basicintegrals}
Let \(c\in \CC_\ast\), \(n\in\NN\) and \([a,b]\subset \RR\) be an interval. Then
\[
\int_a^b x^n \exp (c x) dx = \Bigl.\frac{n!}{c^{n+1}}\Bigl(\sum_{k=0}^n (-1)^k \frac{(c x)^{n-k}}{(n-k)!}\Bigr)\exp\bigl(cx\bigr)\Bigr|_a^b.
\]
\end{lemma}

Next, we would like to compute the {\it Dirichlet energy} of a section \(\tilde\psi\in \Gamma_{pl}(\tilde{\spc L})\), i.e.
\[
\mathrm{E}_D(\tilde\psi) = \int_{\complex X} \bigl|\nabla \tilde\psi\bigr|^2 .
\]
Note, that the Dirichlet energy comes with a corresponding positive-semidefinite hermitian form \(\langle\!\langle.,.\rangle\!\rangle_{\!D}\) - called the {\it Dirichlet product}. Clearly, like the metric, the Dirichlet product is completely determined by the values it takes on concentrated sections. 

In general, if \(\tilde\psi\neq 0\) is piecewise-linear section concentrated at \(i\), it is given on the vertex star \(S_i\) as a product \(\tilde\psi = x_i\,\tilde\phi\), where \(x_i\) denotes the barycentric coordinate of the vertex \(i\) and \(\tilde\phi\) is a local section radially parallel with respect to \(i\). Clearly, 
\[
\nabla\tilde\psi = dx_i\, \tilde\phi + \imath\,x_i\,\omega_i\, \tilde\phi,
\]
where \(\omega_i\) denotes the rotation form of \(\tilde\phi\), i.e. \(\nabla\tilde\phi = \imath \omega_i\, \tilde\phi\). Note here that \(\omega_i\) does not depend on the actual value of \(\tilde\psi\) at \(i\), but is the same for all non-vanishing piecewise-linear sections concentrated at \(i\). 

To compute the rotation form \(\omega_i\) at a given point \(p_0\in S_i\), we use a local section \(\zeta\) which is radially parallel with respect to \(p_0\) such that \(\zeta_{p_0} = \tilde\phi_{p_0}\). Then we can express \(\tilde\phi\) in terms of \(\zeta\), i.e.
\[
\tilde\phi = z\, \zeta,
\] 
for some piecewise-smooth \(\CC_\ast\)-valued function \(z\) defined locally at \(p_0\). Clearly, \(|z|\) is constant, and hence
\[
\left.\imath\omega_i\right|_{p_0}\, \tilde\phi_{p_0} = \left.\nabla\tilde\phi\right|_{p_0} = \left.dz\right|_{p_0}\, \zeta_{p_0} + z(p_0)\, \left.\nabla \zeta\right|_{p_0} = \left.d\log z \right|_{p_0} \tilde\phi_{p_0} = \left.\imath d\arg z\right|_{p_0}\, \tilde\phi_{p_0}.
\]
The clue is that we can now use the relation of parallel transport and curvature to obtain an explicit formula for \(z\). If \(p\) is sufficiently close to \(p_0\), then the three points \(p\), \(i\) and \(p_0\) determine an oriented triangle \(\Delta^p\) which is contained in a simplex of \(\complex X\). Its boundary curve \(\gamma_p\) consists of three line segments \(\gamma_1,\gamma_2,\gamma_3\) connecting \(p\) to \(i\), \(i\) to \(p_0\) and \(p_0\) back to \(p\). Hence on each of these segments either \(\tilde\phi\) of \(\zeta\) is parallel and
\[
\zeta_p = P_{\gamma_p}(\tilde\phi_p) = \exp\bigl(\imath \int_{\Delta^p} \tilde\Omega\bigr)\tilde\phi_p.
\]
Thus we obtain that \(z(p) = \exp\bigl(-\imath \int_{\Delta^p} \tilde\Omega\bigr)\) and hence 
\[
\Bigl.\omega_i\Bigr|_{p_0} = -\Bigl.d\bigl(\int_{\Delta^p} \tilde\Omega\bigr)\Bigr|_{p_0}. 
\]
Now, if \(\Delta^p\) is contained in a simplex \(\sigma=\{i_0,\ldots,i_n\}\), one verifies that
\[
\int_{\Delta^p} dx_{i_j}\wedge dx_{i_k} = \frac{1}{2}\bigl(x_{i_j}(p_0)x_{i_k}(p) - x_{i_k}(p_0)x_{i_j}(p)\bigr).
\]
Thus,
\begin{align*}
\Bigl.d\bigl(\int_{\Delta^p} \tilde\Omega\bigr)\Bigr|_{p_0} &= \sum_{1\leq j< k\leq n} 2\,\Omega_{i_0 i_j i_k}\, \Bigl.d \bigl(\int_{\Delta^p} d x_{i_j}\wedge dx_{i_k}\bigr)\Bigr|_{p_0} \\
 &= \sum_{1\leq j< k\leq n} \Omega_{i_0 i_j i_k}\, \Bigl.\bigl(x_{i_j}dx_{i_k} - x_{i_k}dx_{i_j}\bigr)\Bigr|_{p_0},\\
 &= \sum_{1\leq j\leq n} \Bigl.\bigl(\sum_{k\neq j}\Omega_{i_0 i_j i_k}x_{i_k}\bigr)dx_{i_j}\Bigr|_{p_0}
\end{align*}
and, using the convention on \(\Omega\) from above, we find the following simple formula:
\begin{equation}\label{eq:localrotationform}
\omega_i = \Bigl.\sum_{j}\bigl(\sum_{k}\Omega_{i j k}\, x_k\bigr)\, dx_j \Bigr|_{S_i},
\end{equation}
where we sum over the whole vertex set of \(\complex X\).

Now, given this local form expressions, we can finally return to the computation of the products which we are actually interested in. Therefore we consider two piecewise-linear sections concentrated at the vertices \(i\) and \(j\): 
\[
\tilde\psi^i:=\iota(\psi_i),\quad \tilde\psi^j:=\iota(\psi_j),
\]
for some \(\psi_i\in \spc L_i\) and \(\psi_j\in \spc L_j\). On their common support \(S_i\intersection S_j\) both section can be expressed, just as above, as products of a real-valued piecewise-linear hat functions \(x_i\) and \(x_j\) and radially parallel local sections \(\tilde\phi^i\) and \(\tilde\phi^j\): 
\[\tilde\psi^i = x_i\, \tilde\phi^i, \quad \tilde\psi^j = x_j\, \tilde\phi^j.\] 
Clearly,
\begin{align*}
\langle\!\langle \tilde\psi^j, \tilde\psi^i \rangle\!\rangle_{\!D} 
= & \int_{S_i \intersection S_j} \langle dx_j \tilde\phi^j + \imath x_j\omega_j \, \tilde\phi^j, dx_i \tilde\phi^i + \imath x_i\omega_i \, \tilde\phi^i \rangle\\ 
= & \int_{S_i \intersection S_j} \langle dx_j + \imath x_j\omega_j,dx_i + \imath x_i\omega_i\rangle\,\langle\tilde\phi^j,\tilde\phi^i\rangle.
\end{align*}
With \eqref{eq:integrandproduct} we see that 
\[
\langle\tilde\phi^j,\tilde\phi^i\rangle = \langle \psi_j, \eta_{\ij}(\psi_i) \rangle\,\exp\bigl(-\imath\sum_{m} \Omega_{i j m} x_{m}\bigr).
\] 
Moreover, by \eqref{eq:localrotationform},
\begin{align*}
\langle dx_j + \imath x_j\omega_j,dx_i + \imath x_i\omega_i\rangle = & \Bigl[\langle dx_j,dx_i\rangle+ \sum_{k^\prime,k^\dprime,l^\prime,l^\dprime} \Omega_{ik^\prime l^\prime}\Omega_{jk^\dprime l^\dprime}x_j x_i x_{l^\prime}x_{l^\dprime}\langle dx_{k^\prime},dx_{k^\dprime}\rangle \Bigr]\\
& + \imath \Bigl[\sum_{k^\prime,l^\prime} (\Omega_{ik^\prime l^\prime} x_i x_{l^\prime}\langle dx_j,dx_{k^\prime}\rangle -  \Omega_{jk^\prime l^\prime} x_jx_{l^\prime}\langle dx_{k^\prime},dx_i\rangle)\Bigr].  
\end{align*}
The constants \(\langle dx_{k^\prime}, dx_{l^\prime} \rangle\) are basically provided by the following lemma. 
\begin{lemma}\label{lma:canonicalgradients}
Let \(\sigma= \{v_0,\ldots,v_n\}\) be a Euclidean simplex of dimension \(n>0\) and let \(x_i\) denote its barycentric coordinate functions. Then
\[
\grad x_i = -\frac{1}{h_i} N_i,
\]
where \(h_i\) denotes the distance between \(v_i\) and \(\sigma_i = \sigma\setminus \{v_i\}\) and \(N_i\) denotes the outward-pointing unit normal of \(\sigma_i\).
\end{lemma}

\begin{proof}
This immediately follows from two basic facts: First, \(dx_i (v_j-v_0) = \delta_{\ij}\) for \(i,j> 0\). Second, \(h_i = \langle v_0-v_i, N_i\rangle\).
\end{proof}

\lmaref{lma:canonicalgradients} yields almost immediately a higher dimensional analogue of the well-known cotangent formula for surfaces.

\begin{theorem}[Cotangent Formula]
Let \(\sigma\) be a simplex of a Euclidean simplicial complex \(\complex X\) and let \(\dim \sigma >1\). If \(i\neq j\),
\[
c_\sigma^{\ij}:= \int_{\sigma} \langle dx_i , dx_j \rangle = \begin{cases}\hfill-\frac{1}{n(n-1)} \cot \alpha^{\ij}_\sigma\, \vol \bigl(\sigma\setminus \{i,j\}\bigr),\hfill &\textit{if } \{i,j\} \subset \sigma, \\ \hfill 0\hfill & else.\end{cases}
\]
Here \(\alpha_\sigma^{\ij}\) denotes the angle between the faces \(\sigma\setminus \{i\}\) and \(\sigma\setminus\{j\}\). Moreover,
\[
c_\sigma^{ii}:= \int_{\sigma} |dx_i|^2 = \begin{cases}\hfill\frac{1}{n\, h_i}\, \vol \bigl(\sigma\setminus \{i\}\bigr),\hfill &\textit{if } i \in \sigma, \\ \hfill 0\hfill & else,\end{cases}
\]
where \(h_i\) denotes the distance between the vertex \(i\) and the face \(\sigma\setminus \{i\}\).
\end{theorem}

\begin{proof}
Clearly, if \(\{i,j\}\not\subset \sigma\), then \(\int_{\sigma} \langle dx_i , dx_j \rangle = 0\). Now, let \(\{i,j\}\subset \sigma\), \(i\neq j\). With the notation of \lmaref{lma:canonicalgradients}, we have
\[
\int_{\sigma} \langle dx_i , dx_j \rangle = \langle\grad x_i ,\grad x_j\rangle\, \vol \sigma = \frac{\langle N_i, N_j \rangle}{h_i h_j}\,\vol \sigma.
\]
Furthermore, \(\cos \alpha_\sigma^{\ij}= -\langle N_i, N_j\rangle\) and \(n!\, \vol \sigma = (n-2)!\, h_ih_j \sin\alpha_\sigma^{\ij}\, \vol\bigl(\sigma\setminus \{i,j\}\bigr)\). This yields the first part of the theorem. Similarly, \(n\,\vol\sigma = h_i\,\vol\bigl(\sigma\setminus \{i\}\bigr)\). Setting then \(i=j\) yields the second part.
\end{proof}

\begin{definition}\label{def:canonicaldirichletatoms}
Let \(\complex X\) be an \(n\)-dimensional simplicial manifold and let \(\Omega \in \Omega^2(\complex X, \RR)\). Let \(\sigma\) be an \(n\)-simplex and \(i,j,k,l\) be vertices of \(\complex X\). Then, let
\begin{align*}
\Lambda_{\sigma,i,j}^\Omega & := \frac{1}{\vol(\sigma)}\int_\sigma \exp\bigl(-\imath\sum_m \Omega_{\ij m} x_m\bigr), \\
\Xi_{\sigma,i,j}^\Omega(k,l) & := \frac{1}{\vol(\sigma)}\int_\sigma x_i x_j x_k x_l\, \exp\bigl(-\imath\sum_m \Omega_{\ij m} x_m\bigr),
\end{align*}
where we choose for the integration an arbitrary discrete metric on \(\complex X\).
\end{definition}

\begin{remark}
Just like the functions \(\Theta_{\sigma,i,j}^\Omega\), the values \(\Lambda_{\sigma,i,j}^\Omega\) and the functions \(\Xi_{\sigma,i,j}^\Omega\) and are well-defined (compare \rmkref{rmk:independenceofchoiceofmetric}).
\end{remark}

Now, with these definitions, we can summarize the above discussion by the following theorem.

\begin{theorem}[Discrete Dirichlet Energy]
Let \(\spc L\) be a discrete hermitian line bundle with curvature \(\Omega\) over an \(n\)-dimensional Euclidean simplicial manifold \(\complex X\), then the {\it Dirichlet product} on \(\Gamma(\spc L)\) induced by the associated piecewise-smooth hermitian line bundle is given as follows: If \(\phi=\sum_i \phi_i\) and \(\psi = \sum_i \psi_i\) are two discrete sections,   
\[
\langle\!\langle \phi, \psi \rangle\!\rangle_{\!D} = \sum_{i,j} w_\Omega^{\ij}\, \langle \phi_j, \eta_{\ij}(\psi_i)\rangle,\quad w_\Omega^{\ij} = \sum_{\{i,j\}\subset \sigma \in {\complex X}_n} W_{\sigma,i,j}^\Omega,
\]
where
\begin{align*}
W_{\sigma,i,j}^\Omega = & \Bigl[ c^{\ij}_\sigma\, \Lambda_{\sigma,i,j}^\Omega + \sum_{k^\prime,k^\dprime,l^\prime,l^\dprime} \Omega_{ik^\prime l^\prime}\Omega_{jk^\dprime l^\dprime}\,c_\sigma^{k^\prime k^\dprime}\,\Xi_{\sigma,i,j}^\Omega (l^\prime,l^\dprime)\Bigr] \numberthis\label{eq:dirichletweights}\\
& + \imath\Bigl[ \sum_{k^\prime,l^\prime} \bigl(\Omega_{ik^\prime l^\prime}\,c_\sigma^{j k^\prime}\,\Theta_{\sigma,i,j}^\Omega (i,l^\prime) - \Omega_{jk^\prime l^\prime}\,c_\sigma^{i k^\prime}\,\Theta_{\sigma,i,j}^\Omega (j,l^\prime) \bigr)\Bigr].
\end{align*}
\end{theorem}

\section{Discrete Energies on Surfaces - An Example}

While the computation of the Dirichlet product \(\langle\!\langle.,.\rangle\!\rangle_{\!D}\) and the metric \(\langle\!\langle.,.\rangle\!\rangle\) of discrete sections is quite complicated and tedious for higher dimensional simplicial manifolds, it is manageable for the \(2\)-dimensional case. We are going to compute it explicitly.

Throughout this section let \(\spc L\) denote a discrete hermitian line bundle with curvature \(\Omega\) over a Euclidean simplicial surface \(\complex X\) and let \(\sigma=\{i,j,k\}\) be one of its triangles.

The metric \(\langle\!\langle.,.\rangle\!\rangle\) is easily obtained. We basically just need to compute the values \(\Theta_{\sigma,i,i}^\Omega (i,i)\) and \(\Theta_{\sigma,i,j}^\Omega (i,j)\), which can be done over the standard triangle. We get
\begin{equation}\label{eq:metricintegrals}
\Theta_{\sigma,i,i}^\Omega (i,i) = \frac{1}{6},\quad \Theta_{\sigma,i,j}^\Omega (i,j) = 2\frac{\exp(-\imath \Omega_{\ijk}) - 1 + \imath \Omega_{\ijk} + \tfrac{1}{2}\Omega_{\ijk}^2 - \imath \tfrac{1}{6}\Omega_{\ijk}^3}{\Omega_{\ijk}^4}.
\end{equation}

Now, we compute the Dirichlet product \(\langle\!\langle.,.\rangle\!\rangle_{\!D}\) on \(\complex X\). For \(n=2\), the expressions \(W_{\sigma,i,i}^\Omega\) and  \(W_{\sigma,i,j}^\Omega\) simplify drastically. First, we look at the diagonal terms. We have
\begin{align*}
\sum_{k^\prime,k^\dprime,l^\prime,l^\dprime} & c_\sigma^{k^\prime k^\dprime} \Omega_{ik^\prime l^\prime}\Omega_{ik^\dprime l^\dprime}\,\Xi_{\sigma,i,i}^\Omega(l^\prime,l^\dprime) \hfill\\
 & = \Bigl(c_\sigma^{jj}\Xi_{\sigma,i,i}^\Omega(k,k) - 2c^{jk}_\sigma \Xi_{\sigma,i,i}^\Omega(j,k) + c_\sigma^{kk}\Xi_{\sigma,i,i}^\Omega(j,j) \Bigr)\Omega_{\ijk}^2 ,
\end{align*}
and with
\[
\Lambda_{\sigma,i,i} = 1, \quad  \Xi_{\sigma,i,i}(j,j) = \frac{1}{90} =  \Xi_{\sigma,i,i}(k,k), \quad \Xi_{\sigma,i,i}(j,k) = \frac{1}{180}  
\]
we get the following formula:
\begin{equation*}
W_{\sigma,i,i}^\Omega = c_\sigma^{ii} + \frac{c_\sigma^{jj} - c^{jk}_\sigma + c_\sigma^{kk}}{90}\Omega_{\ijk}^2.
\end{equation*}
Now we would like to obtain a similar formula for the off-diagonal terms. Since \(dx_i+dx_j= -dx_k\), we have \(c_\sigma^{jk} + c_\sigma^{ki}= -c_\sigma^{kk}\). Hence,
\begin{align*}
\sum_{k^\prime,k^\dprime,l^\prime,l^\dprime}& c_\sigma^{k^\prime k^\dprime} \Omega_{ik^\prime l^\prime}\Omega_{jk^\dprime l^\dprime}\,\Xi_{\sigma,i,j}^\Omega(l^\prime,l^\dprime)\\
& = -\Bigl(c_\sigma^{\ij}\Xi_{\sigma,i,j}^\Omega(k,k) + c_\sigma^{kk}\bigl(\Xi_{\sigma,i,j}^\Omega(i,j) + \Xi_{\sigma,i,j}^\Omega(j,k)\bigr)\Bigr)\Omega_{\ijk}^2.
\end{align*}
This time the expressions become more complicated. We get
\begin{equation*}
\resizebox{.95\hsize}{!}{$\Xi_{\sigma,i,j}^\Omega(k,k) = \frac{2}{\Omega_{\ijk}^6}\Bigl( 20 - 12\imath \Omega_{\ijk} - 3\Omega_{\ijk}^2 + \tfrac{1}{3}\imath\Omega_{\ijk}^3 + \bigl(-20 - 8\imath \Omega_{\ijk}+ \Omega_{\ijk}^2\bigr)\exp\bigl(-\imath\Omega_{\ijk}\bigr)\Bigr),$} 
\end{equation*}
\begin{equation*} 
\resizebox{.99\hsize}{!}{$\Xi_{\sigma,i,j}^\Omega(i,j) + \Xi_{\sigma,i,j}^\Omega(j,k) = \frac{2}{\Omega_{\ijk}^6}\Bigl(-6 + 4\imath \Omega_{\ijk} + \Omega_{\ijk}^2 + \tfrac{1}{12}\Omega_{\ijk}^4 - \tfrac{1}{30}\imath\Omega_{\ijk}^5 + \bigl(6 + 2\imath \Omega_{\ijk}\bigr)\exp\bigl(-\imath\Omega_{\ijk}\bigr)\Bigr)$}.
\end{equation*}
Thus,
\begin{align*}
\sum_{k^\prime,k^\dprime,l^\prime,l^\dprime} & c_\sigma^{k^\prime k^\dprime} \Omega_{ik^\prime l^\prime}\Omega_{jk^\dprime l^\dprime}\,\Xi_{\sigma,i,j}^\Omega(l^\prime,l^\dprime) = \\ 
\frac{2}{\Omega_{\ijk}^4} \Bigl( &\bigl[6 c_\sigma^{kk} - 20 c_\sigma^{\ij}\bigr] + \bigl[12 c_\sigma^{\ij} - 4 c_\sigma^{kk}\bigr]\imath\Omega_{\ijk} + \bigl[3 c_\sigma^{\ij} - c_\sigma^{kk}\bigr]\Omega_{\ijk}^2 - \tfrac{c_\sigma^{\ij}}{3}\imath\Omega_{\ijk}^3 - \tfrac{c_\sigma^{kk}}{12}\Omega_{\ijk}^4 \\
& + \tfrac{c_\sigma^{kk}}{30}\imath\Omega_{\ijk}^5 + \bigl(\bigr[20 c_\sigma^{\ij} - 6 c_\sigma^{kk}\bigl] + \bigr[8 c_\sigma^{\ij} - 2 c_\sigma^{kk}\bigl]\imath\Omega_{\ijk} - c_\sigma^{\ij}\Omega_{\ijk}^2\bigr)\exp\bigl(-\imath\Omega_{\ijk}\bigr).
\Bigr)
\end{align*}
Now, let us look at the second sum in \eqref{eq:dirichletweights}. We have
\begin{align*}
\imath\sum_{k^\prime,l^\prime}& \bigl(\Omega_{ik^\prime l^\prime}c_\sigma^{j k^\prime} \Theta_{\sigma,i,j}^\Omega(i,l^\prime) -  \Omega_{jk^\prime l^\prime}c_\sigma^{i k^\prime}\Theta_{\sigma,i,j}^\Omega(j,l^\prime)\bigr)\\
& = \Bigl(c_\sigma^{ii}\Theta_{\sigma,i,j}^\Omega(j,k) + c_\sigma^{jj}\Theta_{\sigma,i,j}^\Omega(k,i) + c_\sigma^{kk}\Theta_{\sigma,i,j}^\Omega(i,j)\Bigr)\imath\Omega_{\ijk}.
\end{align*}
The formula for \(\Theta_{\sigma,i,j}^\Omega(i,j)\) is already given in \eqref{eq:metricintegrals}. Further, we have
\begin{equation*}
\Theta_{\sigma,i,j}^\Omega(j,k) = \frac{2}{\Omega_{\ijk}^4}\Bigl(3 - 2\imath\Omega_{\ijk} - \tfrac{1}{2}\Omega_{\ijk}^2 + \bigl(-3 +\imath\Omega_{\ijk}\bigr)\exp\bigl(-\imath\Omega_{\ijk}\bigr)\Bigr) = \Theta_{\sigma,i,j}^\Omega(k,i).  
\end{equation*}
Thus we get
\begin{align*}
\imath\sum_{k^\prime,l^\prime} \bigl(\Omega_{ik^\prime l^\prime}c_\sigma^{j k^\prime} \Theta_{\sigma,i,j}^\Omega(i,l^\prime) & -  \Omega_{jk^\prime l^\prime}c_\sigma^{i k^\prime}\Theta_{\sigma,i,j}^\Omega(j,l^\prime)\bigr) = \\
\frac{2}{\Omega_{\ijk}^4} \Bigl( \bigl[ 3(c_\sigma^{ii} + c_\sigma^{jj}) - c_\sigma^{kk}\bigr] & \imath\Omega_{\ijk} + \bigl[2(c_\sigma^{ii} + c_\sigma^{jj}) - c_\sigma^{kk} \bigr]\Omega_{\ijk}^2 + \tfrac{1}{2}\bigl[c_\sigma^{kk}- c_\sigma^{ii} - c_\sigma^{jj}\bigr]\imath\Omega_{\ijk}^3\\
+ \tfrac{c_\sigma^{kk}}{6} \Omega_{\ijk}^4 + \bigl( & \bigl[c_\sigma^{kk} - 3(c_\sigma^{ii} + c_\sigma^{jj})\bigr]\imath\Omega_{\ijk} + \bigr[c_\sigma^{ii} + c_\sigma^{jj}\bigl]\Omega_{\ijk}^2\bigr)\exp\bigl(-\imath\Omega_{\ijk}\bigr) 
\Bigr).
\end{align*}
Hence, with 
\[
\Lambda_{
\sigma,i,j}^\Omega = \frac{2}{\Omega_{\ijk}^4}\Bigl(\Omega_{\ijk}^2 - \imath\Omega_{\ijk}^3 - \Omega_{\ijk}^2\exp\bigl(-\imath\Omega_{\ijk}\bigr)\Bigr),
\]
\eqref{eq:dirichletweights} becomes
\begingroup\makeatletter\def\f@size{8}\check@mathfonts
\def\maketag@@@#1{\hbox{\m@th\large\normalfont#1}}
\begin{align*}
W_{\sigma,i,j}^\Omega = \frac{2}{\Omega_{\ijk}^4} \Bigl(& \bigl[ 6\sigma^{kk}-20 c_\sigma^{\ij}\bigr] + \bigl[12c_\sigma^{\ij}+3(c_\sigma^{ii}+c_\sigma^{jj})-5c_\sigma^{kk}\bigr]\imath\Omega_{\ijk} + \bigl[4c_\sigma^{\ij} + 2(c_\sigma^{ii} + c_\sigma^{jj}- c_\sigma^{kk}) \bigr]\Omega_{\ijk}^2\\
+ \tfrac{1}{6}&\bigl[3(c_\sigma^{kk}- c_\sigma^{ii} - c_\sigma^{jj}) - 8 c_\sigma^{\ij}\bigr]\imath\Omega_{\ijk}^3 + \tfrac{1}{12}c_\sigma^{kk} \Omega_{\ijk}^4 + \tfrac{1}{30}c_\sigma^{kk} \Omega_{\ijk}^4\\ + \bigl(&\bigl[20 c_\sigma^{\ij} - 6 c_\sigma^{kk}\bigr] + \bigr[8 c_\sigma^{\ij} - 3(c_\sigma^{ii}+c_\sigma^{jj}) - c_\sigma^{kk}\bigl]\imath\Omega_{\ijk} + \bigl[c_\sigma^{ii}-2c_\sigma^{\ij}+c_\sigma^{jj}]\bigr]\Omega_{\ijk}^2\bigr)\exp\bigl(-\imath\Omega_{\ijk}\bigr) 
\Bigr).
\end{align*}
\endgroup
Since \(n=2\), the weights \(c_\sigma^{\ij}\) are just given as follows:
\[
c_\sigma^{\ij} = -\frac{\cot \alpha_\sigma^{\ij}}{2},\quad c_\sigma^{kk}= \frac{\ell_{\ij}}{2h_k},
\]
where \(\ell_{\ij}\) denotes the edge length. We would like to express them explicitly in terms of the Euclidean metric \(g\) of \(\sigma\). In fact, we can distinguish the vertex \(k\) as origin and use the hat functions \(x_i\) and \(x_j\) as coordinates on \(\sigma\). With respect to these coordinates, the metric is given by a matrix:
\[
g= \begin{pmatrix}g_{11} & g_{12}\\ g_{21} & g_{22} \end{pmatrix}.
\]
In terms of \(g\) the cotangent weights are given as follows:
\begin{align*}
c_\sigma^{\ij} &= -\frac{g_{12}}{2\sqrt{\det g}}, & c_\sigma^{jk} &= -\frac{g_{11} - g_{12}}{2\sqrt{\det g}}, & c_\sigma^{ki} &= -\frac{g_{22} - g_{12}}{2\sqrt{\det g}}, \\
c_\sigma^{kk} &= \frac{g_{11}-2g_{12}+g_{22}}{2\sqrt{\det g}}, & c_\sigma^{ii} &= \frac{g_{22}}{2\sqrt{\det g}}, & c_\sigma^{jj} &= \frac{g_{11}}{2\sqrt{\det g}},
\end{align*}
and we have rederived the formulas in \cite{KC13}:
\begin{align*}
W_{\sigma,i,j}^\Omega = & \frac{1}{\vol(\sigma)\Omega_{\ijk}^4} \Bigl( \bigl[3g_{11} + 4g_{12} + 3g_{22}\bigr] - \bigl[g_{11}+g_{12}+g_{22}\bigr]\imath\Omega_{\ijk} + \tfrac{g_{12}}{6}\imath\Omega_{\ijk}^3 \\ 
& + \tfrac{g_{11}-2g_{12}+g_{22}}{24} \Omega_{\ijk}^4 + \tfrac{g_{11}-2g_{12}+g_{22}}{60}\Omega_{\ijk}^4 - \bigl(\bigl[3g_{11} + 4g_{12} + 3g_{22}\bigr] \\ 
& + \bigr[2g_{11}+3g_{12}+2g_{22}\bigl]\imath\Omega_{\ijk} - \tfrac{1}{2}\bigl[g_{11}+2g_{12}+g_{22}\bigr]\Omega_{\ijk}^2\bigr)\exp\bigl(-\imath\Omega_{\ijk}\bigr)\Bigr).
\end{align*}

\bibliographystyle{abbrv}
\bibliography{dhlb}{}
\newpage

\end{document}